\documentclass[10pt]{article}
 \usepackage{amsmath, amsfonts, amsthm, amssymb, amscd, enumerate, url, slashed, stmaryrd, upgreek, thmtools, mathdots}
 \usepackage{float}
 \usepackage{tikz}
 \usetikzlibrary{matrix}
 
\begingroup
 \makeatletter
 \@for\theoremstyle:=definition,remark,plain\do{%
 \expandafter\g@addto@macro\csname th@\theoremstyle\endcsname{%
 \addtolength\thm@preskip\parskip
 }%
 }
 \endgroup
 
 \usepackage[titletoc,toc,title]{appendix}

 \usepackage[colorlinks = true,
 linkcolor = blue,
 urlcolor = cyan,
 citecolor = red,
 anchorcolor = blue,
 pagebackref]{hyperref}
 
\usepackage[alphabetic,backrefs]{amsrefs}
 
 \usepackage[parfill]{parskip}
 \usepackage{anysize}
 \marginsize{1in}{1in}{1in}{1in}

\declaretheorem[name=Theorem,numberwithin=section]{thm}
\declaretheorem[name=Proposition,numberlike=thm]{prop}

\declaretheorem[name=Corollary,numberlike=thm]{cor}
\declaretheorem[name=Question,style=definition,qed=$\blacktriangle$,numberlike=thm]{quest}

\declaretheorem[name=Definition,style=definition,qed=$\blacktriangle$,numberlike=thm]{defn}
\declaretheorem[name=Example,style=definition,qed=$\blacktriangle$,numberlike=thm]{ex}
\declaretheorem[name=Remark,style=definition,qed=$\blacktriangle$,numberlike=thm]{rmk}

\newcounter{noteCounter}
\DeclareMathOperator\Id{Id}

\DeclareMathOperator{\End}{End}

\DeclareMathOperator{\Hol}{Hol}
\DeclareMathOperator{\hol}{\mathfrak{hol}}

\newcommand{\w}{\wedge}

\newcommand{\vol}{\mathsf{vol}}

\newcommand{\real}{\operatorname{\mathrm{Re}}}
\newcommand{\imag}{\operatorname{\mathrm{Im}}}

\newcommand{\Stab}{\mathrm{Stab}}

\newcommand{\G}{\mathrm{G}_2}
\newcommand{\Spin}[1]{\mathrm{Spin}(#1)}
\newcommand{\GL}[1]{\mathrm{GL}(#1)}

\newcommand{\U}[1]{\mathrm{U}(#1)}
\newcommand{\SU}[1]{\mathrm{SU}(#1)}
\newcommand{\Sp}[1]{\mathrm{Sp}(#1)}
\newcommand{\SO}[1]{\mathrm{SO}(#1)}

\newcommand{\OO}[1]{\mathrm{O}(#1)}

\newcommand{\R}{\mathbb R}
\newcommand{\C}{\mathbb C}

\newcommand{\Qu}{\mathbb H}

\newcommand{\ph}{\varphi}
\newcommand{\ps}{\psi}
\newcommand{\Ph}{\Phi}
\newcommand{\st}{\star}
\newcommand{\hk}{\mathbin{\! \hbox{\vrule height0.3pt width5pt depth 0.2pt \vrule height5pt width0.4pt depth 0.2pt}}}

\newcommand{\del}{\partial}

\newcommand{\ddx}[1]{\frac{\del \, \,}{\del x^{#1}}}

\newcommand{\ddy}[1]{\frac{\del \, \,}{\del y^{#1}}}

\newcommand{\rest}[2]{{#1}|_{#2}}

\newcommand{\wt}{\widetilde}
\newcommand{\wh}{\widehat}

\newcommand{\ol}{\overline}

\newcommand{\rG}{\mathrm{G}}
\newcommand{\rH}{\mathrm{H}}
\newcommand{\frg}{\mathfrak{g}}
\newcommand{\frh}{\mathfrak{h}}

\newcommand{\so}{\mathfrak{so}}
\newcommand{\gl}{\mathfrak{gl}}

\newcommand{\wG}{\widehat{\mathrm{G}}}
\newcommand{\Aut}{\mathrm{Aut}}
\newcommand{\Ad}{\mathrm{Ad}}

\begin{document}

\title{Extrinsic geometry of calibrated submanifolds}

\author{Spiro Karigiannis \\ \emph{Department of Pure Mathematics, University of Waterloo} \\ \tt{karigiannis@uwaterloo.ca} \and Luc\'ia Mart\'in-Merch\'an \\ \emph{Department of Pure Mathematics, University of Waterloo} \\ \tt{lucia.martinmerchan@uwaterloo.ca} }

\maketitle

\begin{abstract}
Given a calibration $\alpha$ whose stabilizer acts transitively on the Grassmanian of calibrated planes, we introduce a nontrivial Lie-theoretic condition on $\alpha$, which we call \emph{compliancy}, and show that this condition holds for many interesting geometric calibrations, including K\"ahler, special Lagrangian, associative, coassociative, and Cayley. We determine a sufficient condition that ensures compliancy of $\alpha$, we completely characterize compliancy in terms of properties of a natural involution determined by a calibrated plane, and we relate compliancy to the geometry of the calibrated Grassmanian.

The condition that a Riemannian immersion $\iota \colon L \to M$ be calibrated is a first order condition. By contrast, its extrinsic geometry, given by the second fundamental form $A$ and the induced tangent and normal connections $\nabla$ on $TL$ and $D$ on $NL$, respectively, is second order information. We characterize the conditions imposed on the extrinsic geometric data $(A, \nabla, D)$ when the Riemannian immersion $\iota \colon L \to M$ is calibrated with respect to a calibration $\alpha$ on $M$ which is both \emph{parallel} and \emph{compliant}. This motivate the definition of an \emph{infinitesimally calibrated} Riemannian immersion, generalizing the classical notion of a superminimal surface in $\R^4$.
\end{abstract}

\tableofcontents

\section{Introduction} \label{sec:introduction}

Let $(M^n, g)$ be a Riemannian manifold, and let $L^k$ be a submanifold with induced metric $\rest{g}{L}$. The classical Gauss--Codazzi--Ricci equations relate three geometric quantities: the intrinsic geometry of $M$ (that is, the curvature of $g$), the intrinsic geometry of $L$ (the curvature of $\rest{g}{L}$), and the \emph{extrinsic} geometry of the immersion of $L$ in $M$, encoded by the second fundamental form $A$, the induced tangent connection $\nabla$ on $TL$, and the induced normal connection $D$ on the normal bundle $NL$.

Suppose $(M, g)$ has \emph{special Riemannian holonomy} $\rG$. This is encoded by the existence of a torsion-free $\rG$-structure for some Lie group $\rG \subset \SO{n}$, and is characterized by the existence of certain parallel tensors. These always include parallel differential forms which are calibrations, yielding a distinguished class of calibrated submanifolds. The Ambrose--Singer Theorem tells us that the curvature $\ol{R}$ of the Levi-Civita connection $\ol{\nabla}$ of $g$ lies in the Lie algebra $\frg$ of the holonomy group, so the curvature of $g$ is special. The Harvey--Lawson Theorem tells us that a calibrated submanifold is minimal, so its second fundamental form is special. Thus in this case we might expect the Gauss--Codazzi--Ricci equations to tell us more about the relations between the intrinsic geometry of $(M, g)$, the intrinsic geometric of $(L, \rest{g}{L})$, and the extrinsic geometry of $L$ in $M$.

In this paper we examine this situation. Given a torsion-free $\rG$-structure on $(M, g)$, and a calibration $k$-form $\alpha$ stablized by $\rG$, there is determined a subgroup $\rH \subset \rG$ which stabilizes any $\alpha$-calibrated plane. (We need assume that $\rG$ acts transitively on the $\alpha$-calibrated Grassmanian, which is the case in all geometric examples of interest to us.) In this setting, we define a Lie-theoretic condition, which we call \emph{compliancy}, and which we show holds for all of our examples. We investigate the geometric meaning of compliancy and establish a sufficient condition for compliancy that is easy to check, as and we characterize compliancy in terms of a natural involution determined by a calibrated plane.

We prove that, if the compliant calibration is parallel (actually, we impose a slightly stronger condition, which again is always satisfied for a torsion-free $\rG$-structure, see Definition~\ref{defn:parallel-compliant}), and if $L$ is $\alpha$-calibrated, then its extrinsic geometric data $(A, \nabla, D)$ satisfy the following relations. Let $\frg$ and $\frh$ be the Lie algebras of $\rG$ and $\rH$, respectively, let $\wh{\nabla} = \nabla \oplus D$ be the direct sum connection on $\rest{TM}{L} = TL \oplus NL$, and let $B$ be determined from $A$ by $B_u (v, \xi) = (- A_u^* \xi, A_u v)$ for all $u, v \in \Gamma(TL)$ and $\xi \in \Gamma(NL)$. Then we prove in Theorem~\ref{thm:main} that
\begin{equation} \label{eq:intro}
\hol(\wh{\nabla}) \subset \frh \qquad \text{and} \qquad B_u \in \frg \quad \text{for all $u \in \Gamma(TL)$.}
\end{equation}
We show explicitly what these conditions say in the case of the K\"ahler, special Lagrangian, associative, coassociative, and Cayley calibrations. The results in the K\"ahler case are of course classical, but the results for the other calibrated geometries do not seem to be as well-known.

Finally, we consider the question of when one could establish a converse. That is, (stated somewhat imprecisely here), if~\eqref{eq:intro} holds for a submanifold $L$ of $(M, g)$, then under what additional hypotheses could one conclude that the submanifold is in fact calibrated?

\textbf{Notation and conventions.} We identify vectors and covectors throughout using the inner product, so all indices are subscripts. In particular, this means $\Lambda^2 (\R^n) = \so(n)$. We sum over repeated indices. The Lie algebra action of a Lie group $\rG \subset \GL{n}$ with Lie algebra $\frg \subset \gl(n)$ on the space $\Lambda^k (\R^n)$ of $k$-forms is denoted by $\diamond$. That is, if $X = X_{ij} e_i \otimes e_j \in \frg$, and $\alpha \in \Lambda^k (\R^n)$, then
\begin{equation} \label{eq:diamond}
X \diamond \alpha = X_{ij} e_i \w (e_j \hk \alpha).
\end{equation}
This implies that, for $\alpha = \frac{1}{k!} \alpha_{i_1 \cdots i_k} e_{i_1} \w \cdots \w e_{i_k}$, we have $X \diamond \alpha = \frac{1}{k!} (X \diamond \alpha)_{i_1 \cdots i_k} e_{i_1} \w \cdots \w e_{i_k}$, where
$$ (X \diamond \alpha)_{i_1 \cdots i_k} = X_{i_1 p} \alpha_{p i_2 \cdots i_k} + X_{i_2 p} \alpha_{i_1 p i_3 \cdots i_k} + \cdots + X_{i_k p} \alpha_{i_1 \cdots i_{k-1} p}. $$

Given a connection $\nabla$, we use $\Hol(\nabla)$ to denote its holonomy group, and $\Hol^0(\nabla)$ to denote its restricted holonomy group (along null-homotopic loops), which is the identity component of $\Hol(\nabla)$. Both of these groups have the same Lie algebra $\hol(\nabla)$, which is the holonomy algebra of $\nabla$.

\textbf{Acknowledgements.} The seeds of this project were first planted in 2014 during some preliminary conversations between the first author and Jonathan Herman. The first author thanks Daren Cheng and Jesse Madnick for useful discussions. The research of the first author is supported by NSERC.

\section{Calibrated geometry and compliant calibrations} \label{sec:calibrated-main}

We first review some notions of calibrated geometry, including the most important geometric examples. In the process, we identify a particular condition on a calibration form, which we call compliancy, that is a crucial ingredient for the proof of our main theorem. We investigate the geometric meaning of compliancy and show that this condition holds in all interesting cases.

\subsection{Important geometric examples of calibrations} \label{sec:calibrated-review}

Equip $\R^n$ with the standard Euclidean inner product $\langle \cdot, \cdot \rangle$ and orientation. We identify vectors and covectors throughout using $\langle \cdot, \cdot \rangle$. Let $1 \leq k \leq n - 1$. Given an oriented $k$-plane $W \subset \R^n$, we obtain a volume form $\vol_{W} \in \Lambda^k W^*$ on $W$ induced from the given orientation and the restriction of $\langle \cdot, \cdot \rangle$ to $W$. Thus $\vol_{W} (e_1, \ldots, e_k) = 1$ if and only if $\{ e_1, \ldots, e_k \}$ is an oriented orthonormal basis for $W$.

Let $\alpha \in \Lambda^k (\R^n)$. Since $\vol_{W}$ spans the $1$-dimensional space $\Lambda^k W^*$, we have $\rest{\alpha}{W} = \lambda(W) \vol_{W}$ for some $\lambda(W) \in \R$. We say that $\alpha$ is a \emph{calibration} if $\lambda(W) \leq 1$ for all oriented $k$-planes $W \subset \R^n$, with equality on at least one $W$. This is often expressed as $\rest{\alpha}{W} \leq \vol_{W}$ for any oriented $k$-plane $W \subset \R^n$, with equality on at least one $W$.

Thus the property of the $k$-form $\alpha$ being a calibration is that $| \alpha(e_1, \ldots, e_k) | \leq 1$ whenever $\{ e_1, \ldots, e_k \}$ is an orthonormal set in $\R^n$, with equality on at least one such set. Since the space of oriented $k$-planes in $\R^n$ is compact, it is clear that if $\alpha$ is nonzero, we can always scale it to make $\alpha$ a calibration. [This condition on $\alpha$ is also called ``having comass one''.]

If $\alpha \in \Lambda^k (\R^n)$ is a calibration, then an oriented $k$-plane $W \subset \R^n$ is called \emph{calibrated} with respect to $\alpha$ if $\rest{\alpha}{W} = \vol_{W}$.

All the calibrations of interest to us arise from the notion of a $\rG$-structure on $\R^n$, where $\rG$ is a particular Lie subgroup of $\SO{n}$ which is the stabilizer in $\SO{n}$ of a finite set $\{ \alpha_0, \ldots, \alpha_N \}$ of skew-symmetric forms on $\R^n$, including the calibration $\alpha = \alpha_0$ of interest. That is, if $P \in \rG$, then $P^* \alpha = \alpha$. In many cases, but not all, the group $\rG$ is \emph{precisely} the stabilizer of $\alpha$, but this is not relevant for us.

We are primarily concerned with the following five particular examples which arise in the study of Riemannian manifolds with special holonomy. Proofs that these are indeed calibrations can be found in Harvey--Lawson~\cite{HL}.

\begin{ex} \label{ex:Kahler}
Let $n=2m$. Identify $\R^{2m} = \R^m \oplus \R^m = \C^m$, with real coordinates $x^1, \ldots, x^m, y^1, \ldots, y^m$ and associated complex coordinates $z^j = x^j + i y^j$ for $1 \leq j \leq m$. The $2$-form
$$ \omega = \sum_{j=1}^m dx^j \w dy^j = \frac{i}{2} \sum_{j=1}^m dz^j \w d\ol{z}^j $$
is the standard \emph{K\"ahler form} on $\C^m$. For any $1 \leq p \leq m$, the $2p$-form $\frac{1}{p!} \omega^p$ is a calibration on $\C^m$, called the \emph{K\"ahler calibration}. An oriented $2p$-plane $W$ is calibrated with respect to $\frac{1}{p!} \omega^p$ if and only if $W$ is a \emph{complex} subspace of complex dimension $p$, with the canonical orientation induced from the complex structure. An example of a calibrated $2p$-plane $W$ is the complex $p$-plane given by the oriented orthonormal basis $\{ \ddx{1}, \ddy{1}, \ldots, \ddx{p}, \ddy{p} \}$. In this case $\rG = \U{m}$ is the stabilizer in $\SO{2m}$ of $\omega$.
\end{ex}

\begin{ex} \label{ex:slag}
Let $\R^{2m} = \C^m$ as in Example~\ref{ex:Kahler}. Let $\Upsilon = dz^1 \w \cdots \w dz^m$ be the canonical complex volume form, which is a complex-valued $m$-form of type $(m,0)$. The real $m$-form $\real (\Upsilon)$ is a calibration on $\C^m$. An oriented $m$-plane $W$ which is calibrated with respect to $\real (\Upsilon)$ is called a \emph{special Lagrangian plane}. An example of a special Lagrangian $m$-plane $W$ is the plane given by the oriented orthonormal basis $\{ \ddx{1}, \ldots, \ddx{m} \}$. In this case $\rG = \SU{m} \subset \SO{2m}$ is the stabilizer of $\omega, \real \Upsilon, \imag \Upsilon$.
\end{ex}

\begin{ex} \label{ex:assoc}
Let $n = 7$, and let $x^1, \ldots, x^7$ be the standard coordinates on $\R^7$. The $3$-form
\begin{equation} \label{eq:assoc-form}
\begin{aligned}
\ph & = dx^1 \w dx^2 \w dx^3 - dx^1 \w (dx^4 \w dx^5 + dx^6 \w dx^7) \\
& \qquad {} - dx^2 \w (dx^4 \w dx^6 + dx^7 \w dx^5) - dx^3 \w (dx^4 \w dx^7 + dx^5 \w dx^6)
\end{aligned}
\end{equation}
is a calibration on $\R^7$, called the \emph{associative calibration}. An oriented $3$-plane $W$ which is calibrated with respect to $\ph$ is called an \emph{associative plane}. An example of an associative $3$-plane $W$ is the plane given by the oriented orthonormal basis $\{ \ddx{1}, \ddx{2}, \ddx{3} \}$. In this case $\rG = \G \subset \SO{7}$ is the stabilizer of $\ph$.
\end{ex}

\begin{ex} \label{ex:coassoc}
Let $n = 7$ and let $\ph$ be as in Example~\ref{ex:assoc}. The $4$-form
\begin{equation} \label{eq:coassoc-form}
\begin{aligned}
\ps = \st \ph & = dx^4 \w dx^5 \w dx^6 \w dx^7 - dx^2 \w dx^3 \w (dx^4 \w dx^5 + dx^6 \w dx^7) \\
& \qquad {} - dx^3 \w dx^1 \w (dx^4 \w dx^6 + dx^7 \w dx^5) - dx^1 \w dx^2 \w (dx^4 \w dx^7 + dx^5 \w dx^6)
\end{aligned}
\end{equation}
is a calibration on $\R^7$, called the \emph{coassociative calibration}. An oriented $4$-plane $W$ which is calibrated with respect to $\ps$ is called a \emph{coassociative plane}. An example of a coassociative $4$-plane $W$ is the plane given by the oriented orthonormal basis $\{ \ddx{4}, \ddx{5}, \ddx{6}, \ddx{7} \}$. In this case $\rG = \G$ is the stabilizer in $\SO{7}$ of $\ps$.
\end{ex}

\begin{ex} \label{ex:Cayley}
Let $n = 8$, and let $x^0, x^1, \ldots, x^7$ be the standard coordinates on $\R^8$. Let $\ph, \ps$ be as in Examples~\ref{ex:assoc} and~\ref{ex:coassoc}. The $4$-form
\begin{equation} \label{eq:Cayley-form}
\begin{aligned}
\Ph & = dx^0 \w \ph + \ps \\
& = dx^0 \w dx^1 \w dx^2 \w dx^3 - (dx^0 \w dx^1 + dx^2 \w dx^3) \w (dx^4 \w dx^5 + dx^6 \w dx^7) \\
& \qquad {} - (dx^0 \w dx^2 + dx^3 \w dx^1) \w (dx^4 \w dx^6 + dx^7 \w dx^5) \\
& \qquad {} - (dx^0 \w dx^3 + dx^1 \w dx^2) \w (dx^4 \w dx^7 + dx^5 \w dx^6) + dx^4 \w dx^5 \w dx^6 \w dx^7
\end{aligned}
\end{equation}
is a calibration on $\R^8$, called the \emph{Cayley calibration}. An oriented $4$-plane $W$ which is calibrated with respect to $\Ph$ is called a \emph{Cayley plane}. Note that in this case $\st \Phi = \Ph$. An example of a Cayley $4$-plane $W$ is the plane given by the oriented orthonormal basis $\{ \ddx{0}, \ddx{1}, \ddx{2}, \ddx{3} \}$. In this case $\rG = \Spin{7} \subset \SO{8}$ is the stabilizer of $\Ph$.
\end{ex}

\subsection{Compliant calibrations} \label{sec:compliant}

Fix a calibration $k$-form $\alpha$ on $\R^n$, corresponding to a $\rG$ structure on $\R^n$, where $\rG$ is the stabilizer in $\SO{n}$ of a finite set of skew-symmetric forms $\{ \alpha_0, \alpha_1, \ldots, \alpha_N \}$ on $\R^n$, including $\alpha = \alpha_0$. That is,
\begin{equation} \label{eq:G-defn}
\rG = \{ P \in \SO{n} : P^* \alpha_j = \alpha_j, \, \text{for $0 \leq j \leq N$} \}.
\end{equation}

\begin{defn} \label{defn:H}
Let $W \subset \R^n$ be an oriented $k$-plane calibrated with respect to $\alpha$. The subgroup $\rH_W = \Stab(W)$ is the stabilizer in $\rG$ of $W$. That is,
$$ \rH_W = \{ P \in \rG : P(W) = W \}. \qedhere $$
\end{defn}

Note that with respect to the orthogonal direct sum decomposition $\R^n = W \oplus W^{\perp}$, and the orientation on $W^{\perp}$ induced from the chosen orientation on $W$ and the standard orientation on $\R^n$, it is clear that
\begin{equation} \label{eq:H-description}
\rH_W = \rG \cap \big( \SO{W} \times \SO{W^{\perp}} \big).
\end{equation}

Suppose that $\rG$ \emph{acts transitively} on the $\alpha$-calibrated $k$-planes. If $W, W_0$ are two $\alpha$-calibrated $k$-planes, then $\rH_W = P \rH_{W_0} P^{-1}$, where $P \in \rG$ satisfies $W = P W_0$. Thus in this case, we can consider $\rH$ as a well-defined conjugacy class of subgroups of $\rG$, and the Grassmanian $\mathrm{Gr}_{\alpha}$ of $\alpha$-calibrated $k$-planes can be identified with $\rG/\rH$.

$$ \text{From now on, \emph{we always assume that} $\rG$ acts transitively on $\mathrm{Gr}_{\alpha}$.} $$

All of the interesting geometric calibrations discussed in Examples~\ref{ex:Kahler}--\ref{ex:Cayley} have this important property, so in all these cases we have a well-defined conjugacy class $\rH$ of subgroups of $\rG$, which we can identify with the stabilizer in $\rG$ of a particular standard $\alpha$-calibrated $k$-plane $W_0$. The actual $\rG$ and $\rH_{W_0}$ for these examples are given in Table~\ref{table:GH}. Details are in Harvey--Lawson~\cite{HL}.

\begin{table}[ht] \label{table:GH}
\centering{
\renewcommand{\arraystretch}{1.2}
\begin{tabular}{|l|l|l|l|l|l|} \hline
$n$ & $k$ & Calibration & $\alpha$ & $\rG$ & $\rH = \rH_{W_0}$ \\ \hline
$2m$ & $2p$ & K\"ahler & $\frac{1}{p!} \omega^p$ & $\U{m}$ & $\U{p} \times \U{m-p}$ \\ \hline
$2m$ & $m$ & special Lagrangian & $\real \Upsilon$ & $\SU{m}$ & $\SO{m}$ \\ \hline
$7$ & $3$ & associative & $\ph$ & $\G$ & $\SO{4} = \Sp{1}^2 / \{ \pm 1 \}$ \\ \hline
$7$ & $4$ & coassociative & $\ps$ & $\G$ & $\SO{4} = \Sp{1}^2 / \{ \pm 1 \}$ \\ \hline
$8$ & $4$ & Cayley & $\Ph$ & $\Spin{7}$ & $\Sp{1}^3 / \{ \pm 1 \}$ \\ \hline
\end{tabular}
\caption{The groups $\rG$ and $\rH$ for particular interesting geometric calibrations.}
}
\end{table}

For the calibrations given in Table~\ref{table:GH}, the embedding of the subgroup $\rH$ in $\rG$ is as follows:
\begin{itemize} \setlength\itemsep{-1mm}
\item In the K\"ahler case, we have the obvious embedding $\U{p} \times \U{m-p} \subset \U{m}$.
\item In the special Lagrangian case, $\SO{m}$ is embedded in $\SU{m}$ as the real elements, that is, $\SO{m} = \{ P \in \SU{m} : \ol{P} = P \}$. Under the usual identification of $\GL{m, \C}$ with a subgroup of $\GL{2m, \R}$, given by $P + i Q \leftrightarrow \left( \begin{smallmatrix} P & - Q \\ Q & P \end{smallmatrix} \right)$, this says that $\SO{m}$ is the diagonal $\{ \left( \begin{smallmatrix} P & 0 \\ 0 & P \end{smallmatrix} \right) : P \in \SO{m} \}$.
\item In the associative and coassociative cases, an element $\pm (p_1, p_2) \in \SO{4} = \Sp{1}^2 / \{ \pm 1 \}$ acts on $\R^7 = \imag \Qu \oplus \Qu$ by
$$ \pm (p_1, p_2) (v, a) = (p_1 v \ol{p}_1, p_1 a \ol{p}_2), $$
which embeds $\SO{4}$ as a subgroup of $\G \subset \SO{7}$.
\item In the Cayley case, an element $\pm (p_1, p_2, p_3) \in \Sp{1}^3 / \{ \pm 1 \}$ acts on $\R^8 = \Qu \oplus \Qu$ by
$$ \pm (p_1, p_2, p_3) (a, b) = (p_1 a \ol{p}_2, p_1 b \ol{p}_3), $$
which embeds $\Sp{1}^3 / \{ \pm 1 \}$ as a subgroup of $\Spin{7} \subset \SO{8}$.
\end{itemize}
Note that $\rG$ and $\rH$ are always \emph{Lie groups}, because they are topologically closed subgroups of $\SO{n}$.

Let $\frg$ be the Lie algebra of $\rG$, and let $\frh$  be the Lie algebra of $\rH$. The inclusions $\rH \subseteq \rG \subseteq \SO{n}$ yields the inclusions
$$ \frh \subseteq \frg \subseteq \so(n) = \Lambda^2 (\R^n). $$
Since the Lie algebra of $\SO{W} \times \SO{W^{\perp}}$ is $\Lambda^2 W \oplus \Lambda^2 W^{\perp}$, equations~\eqref{eq:G-defn} and~\eqref{eq:H-description} imply that
\begin{equation} \label{eq:gh-defn}
\begin{aligned}
\frg & = \{ X \in \Lambda^2(\R^n) : X \diamond \alpha_j = 0, \, \text{for $0 \leq j \leq N$} \}, \\
\frh & = \frg \cap (\Lambda^2 W \oplus \Lambda^2 W^{\perp}),
\end{aligned}
\end{equation}
where we use the $\diamond$ notation of~\eqref{eq:diamond}.

We have an orthogonal decomposition
\begin{equation} \label{eq:W-decomp}
\so(n) = \Lambda^2 (\R^n) = \Lambda^2 (W \oplus W^{\perp}) = (\Lambda^2 W \oplus \Lambda^2 W^{\perp}) \oplus (W \otimes W^{\perp}).
\end{equation}
Let $\frh^{\perp_{\frg}}$ denote the orthogonal complement of $\frh$ \emph{in} $\frg$, so $\frg = \frh \oplus \frh^{\perp_{\frg}}$. That is,
$$ \frg = \big( \frg \cap (\Lambda^2 W \oplus \Lambda^2 W^{\perp}) \big) \oplus \frh^{\perp_{\frg}}. $$
However, it is \emph{in general not true} that $\frh^{\perp_{\frg}} = \frg \cap (W \otimes W^{\perp})$. It is always true that
\begin{equation} \label{eq:always}
\frh^{\perp_{\frg}} \supseteq \frg \cap (W \otimes W^{\perp})
\end{equation}
because $W \otimes W^{\perp}$ is orthogonal to $\Lambda^2 W \oplus \Lambda^2 W^{\perp}$. This observation motivates the following.

\begin{defn} \label{defn:compliant}
We say that the calibration $\alpha$ is \emph{compliant} if
\begin{equation} \label{eq:compliant}
\frh^{\perp_{\frg}} = \frg \cap (W \otimes W^{\perp})
\end{equation}
where $\frg$ and $\frh$ are as defined above.
\end{defn}

Strictly speaking, the condition of compliancy depends on $W$, because $\frh$ depends on $W$. The next result shows that if $G$ acts transitively on the $\alpha$-calibrated Grassmanian $\mathrm{Gr}_{\alpha}$ (which we are always assuming), then compliancy holds for every $W \in \mathrm{Gr}_{\alpha}$ if and only if it holds for some $W$.

\begin{prop} \label{prop:compliancy-well-defined}
If $\rG$ acts transitively on the $\alpha$-calibrated Grassmanian $\mathrm{Gr}_{\alpha}$ and compliancy holds for some $W_0$, then compliancy holds for every $W$.
\end{prop}
\begin{proof}
Let $P \in G$ be such that $W = P W_0$. Then
$$ \SO{W} \times \SO{W^\perp} = P(\SO{W_0} \times \SO{W^\perp_0}) P^{-1}, \qquad W \otimes W^{\perp} = P( W_0 \otimes W_0^{\perp}) P^{-1}. $$
The first equality says $\rH_W = P \rH_{W_0} P^{-1}$ and thus $\frh_W = P \frh_{W_0} P^{-1}$. Taking both observations into account, combined with the equality $\frg = P \frg P^{-1}$, we obtain
$$ \frg = P \big( \frh_{W_0} \oplus \frg \cap (W_0 \otimes W_0^\perp) \big) P^{-1} = \frh_{W} \oplus \big( \frg \cap (W \otimes W^\perp) \big). $$
Hence, we conclude that $\frh_{W}^{\perp_\frg} = \frg \cap (W \otimes W^\perp)$.
\end{proof}

\begin{rmk} \label{rmk:complicancy-Lie-theoretic}
In fact, the compliancy condition is strictly Lie-theoretic, so it can be phrased for pairs of Lie group representations, without reference to calibrations. That is, given a $\rG$-representation on $\R^n$ and a subgroup $\rH$ of $\rG$, we can say that an $\rH$-representation $W \subset \R^n$ is compliant if~\eqref{eq:compliant} holds. A calibration $\alpha$ whose stabilizer $\rG$ acts transitively on the Grassmannian $\rG/\rH$ of $k$-planes is compliant provided that the $\rH$-representation on a calibrated plane is compliant.
\end{rmk}

In the remainder of this section, we investigate the compliancy condition. First we need to introduce some notation. From the splitting $\R^n = W \oplus W^{\perp}$, we obtain
$$ \Lambda^k (\R^n) = \Lambda^k (W \oplus W^{\perp}) = \bigoplus_{p+q=k} (\Lambda^p W) \otimes (\Lambda^{k-p} W^{\perp}). $$
Elements of $(\Lambda^p W) \otimes (\Lambda^q W^{\perp})$ are said to be of type $(p,q)$ with respect to $W$. Thus for any $\beta \in \Lambda^k (\R^n)$ and any $k$-plane $W \subset \R^n$, we can uniquely decompose
$$ \beta = \sum_{p+q = k} \beta^{p,q,W} \qquad \text{ where $\beta^{p,q,W}$ is of type $(p, q)$ with respect to $W$.} $$
If the $k$-plane $W$ is understood, we often abbreviate $\beta^{p,q,W}$ to $\beta^{p,q}$. Note that
\begin{equation} \label{eq:k0-calibrated}
\text{an oriented $k$-plane $W$ is $\alpha$-calibrated \,} \iff \, \alpha^{k,0,W} = \vol_W.
\end{equation}

\begin{prop} \label{prop:even-normals}
Let $\rG$ be as in~\eqref{eq:G-defn} with $\alpha = \alpha_0$ a calibration. Suppose that for any oriented $k$-plane $W$ which is $\alpha$-\emph{calibrated}, the following condition holds. For each $0 \leq j \leq N$, we have that either:
\begin{equation} \label{eq:even-normals}
\begin{aligned}
\alpha_j & \in \bigoplus_{0 \leq 2r \leq k} (\Lambda^{k-2r} W) \otimes (\Lambda^{2r} W^{\perp}), \quad \text{or} \\
\alpha_j & \in \bigoplus_{0 \leq 2r + 1 \leq k} (\Lambda^{k-(2r+1)} W) \otimes (\Lambda^{2r+1} W^{\perp}).
\end{aligned}
\end{equation}
That is, with respect to $W$, each term in the $(p,q)$ decomposition of $\alpha_j$ has the same parity of $q$. Note that by~\eqref{eq:k0-calibrated} only the first (even) case can occur for $\alpha = \alpha_0$.

Then $\alpha$ is compliant.
\end{prop}
\begin{proof}
Let $X \in \frh^{\perp_{\frg}} \subseteq \frg \subseteq \so(n)$, and write $X = X_1 + X_2$ with $X_1 \in \Lambda^2 W \oplus \Lambda^2 W^{\perp}$ and $X_2 \in W \otimes W^{\perp}$ with respect to the decomposition~\eqref{eq:W-decomp}. Thus the action of $X_1$ preserves both $W$ and $W^{\perp}$, while $X_2$ swaps them. It follows that if $\alpha^{p,q}_j \in (\Lambda^p W) \otimes (\Lambda^q W^{\perp})$, then the Lie algebra action $X_i \diamond \alpha_j$ satisfies
\begin{align*}
X_1 \diamond \alpha^{p,q}_j & \in (\Lambda^p W) \otimes (\Lambda^q W^{\perp}), \\
X_2 \diamond \alpha^{p,q}_j & \in (\Lambda^{p-1} W) \otimes (\Lambda^{q+1} W^{\perp}) \, \oplus \, (\Lambda^{p+1} W) \otimes (\Lambda^{q-1} W^{\perp}).
\end{align*}
From these equations and the hypothesis~\eqref{eq:even-normals} we deduce that the two forms $X_1 \diamond \alpha_j$ and $X_2 \diamond \alpha_j$ lie in orthogonal spaces. Since $X \in \frg$, by~\eqref{eq:gh-defn} we have
$$ 0 = X \diamond \alpha_j = X_1 \diamond \alpha_j + X_2 \diamond \alpha_j. $$
Therefore, $X_1 \diamond \alpha_j = 0$ and $X_2 \diamond \alpha_j = 0$ for all $0 \leq j \leq N$. From~\eqref{eq:gh-defn} we obtain $X_1 \in \frh$ and $X_2 \in \frg$. Since $X \in \frh^{\perp_{\frg}}$, we get $0 = \langle X, X_1 \rangle = \|X_1\|^2$ so $X_1 = 0$ and thus $X = X_2 \in \frg \cap (W \otimes W^{\perp})$. We have shown that
$$ \frh^{\perp_{\frg}} \subseteq \frg \cap (W \otimes W^{\perp}). $$
Since~\eqref{eq:always} always holds, we conclude that~\eqref{eq:compliant} holds, so $\alpha$ is compliant.
\end{proof}

\begin{defn} \label{defn:involution}
Fix $W \in \mathrm{Gr}_{\alpha}$. Define the map
\begin{equation} \label{eq:involution}
\phi = \Id_W \times (-\Id_{W^\perp}),
\end{equation}
which is an involution and lies in $\OO{n}$.
\end{defn}

In general, $\phi \notin \SO{n}$, so it does not make sense to ask if $\phi \in \rG$. However, if $n-k$ is even and $\phi^* \alpha_j = \alpha_j$ for all $0 \leq j \leq N$, then $\phi \in \rG$. Similarly if $k$ is even and $(-\phi)^* \alpha_j = \alpha_j$ for all $0 \leq j \leq N$, then $- \phi \in \rG$. These observations motivate us to define the following ``enhancement'' of $\rG$.

\begin{defn} \label{defn:hatG}
Define the subgroup $\wG \subseteq \OO{n}$ to be the stabilizer (up to sign) in $\OO{n}$ of $\alpha_0, \ldots, \alpha_N$. That is,
\begin{equation} \label{eq:hatG-defn}
\wG =  \{ P \in \OO{n} : P^* \alpha_j = \pm \alpha_j, \, \text{for $0 \leq j \leq N$}\}.
\end{equation}
Of course, $\rG \subseteq \wG$ and the Lie algebra of $\wG$ coincides with $\frg$. 
\end{defn}

\begin{prop} \label{prop:even-normals-involution}
Suppose that $\phi \in \wG$. Then $\alpha$ is compliant.
\end{prop}
\begin{proof}
If $\alpha_j = \sum_{r=0}^k \alpha^{k-r,r}_j$, then $\phi^*\alpha_j = \sum_{r=0}^k (-1)^r\alpha^{k-r,r}_j$. Thus $\phi^* \alpha_j = \pm \alpha_j$ if and only if~\eqref{eq:even-normals} holds for $\alpha_j$. The result now follows from Proposition~\ref{prop:even-normals}.
\end{proof}

In fact, the condition of compliancy can be completely characterized in terms of this involution $\phi$. For that purpose, observe that the map $\Ad(\phi) \in \Aut(\so(n))$ given by $\Ad(\phi)(X) = \phi X \phi^{-1}$ is an involution and that its $\pm 1$-eigenspaces $\so(n)_{\pm 1}$ are
\begin{equation} \label{eq:sonpm}
\so(n)_1= \Lambda^2 W \oplus \Lambda^2 W^\perp, \quad \so(n)_{-1}= W \otimes W^\perp.
\end{equation}
Using~\eqref{eq:gh-defn} and Definition~\ref{defn:compliant}, we have
\begin{equation} \label{eq:compliant-involution}
\text{$\frh = \frg \cap \so(n)_1$ always holds, \quad and $\alpha$ is compliant if and only if $\frh^{\perp_{\frg}}= \frg \cap \so(n)_{-1}$.}
\end{equation}

\begin{thm} \label{thm:compliant-involution}
Let $W \in \mathrm{Gr}_{\alpha}$, and let $\phi$ be the involution~\eqref{eq:involution} associated to $W$. Then $\alpha$ is compliant (with respect to $W$) if and only if $\Ad(\phi)(\frg)= \frg$. That is, if and only if $\Ad(\phi) \in \Aut(\frg)$.
\end{thm}
\begin{proof}
If $\alpha$ is compliant then we have $\frg = \frg\cap \so(n)_1 \oplus \frg \cap \so(n)_{-1}$, and this implies that $\Ad(\phi)(\frg)\subseteq \frg$, since $\Ad(\phi)$ is $\pm 1$ on $\frg \cap \so(n)_{\pm 1}$.

Conversely, suppose that $\Ad(\phi)(\frg)= \frg$. Since $\Ad(\phi)$ is an involution which preserves the inner product of $\so(n)$, we have an orthogonal decomposition
$$ \frg = \frg_1 \oplus \frg_{-1}. $$
We claim that $\frg_1=\frh$. The inclusion $\frh \subseteq \frg_1$ is immediate from~\eqref{eq:gh-defn} and~\eqref{eq:sonpm}, whereas $\frg_1 \subseteq \frg \cap \so(n)_1 = \frh$ from~\eqref{eq:compliant-involution}. It follows that $\frh^{\perp_\frg} = \frg_{-1}$. But $\frg_{-1} \subseteq \frg \cap \so(n)_{-1}$, and by~\eqref{eq:always} we also have $\frg \cap \so(n)_{-1} \subseteq \frh^{\perp_\frg}$. Thus we deduce that $\frh^{\perp_\frg} = \frg \cap \so(n)_{-1}$, so $\alpha$ is compliant.
\end{proof}

Note that since the Lie algebra of $\wG$ is $\frg$, if $\phi \in \wG$ then $\Ad(\phi) \in \Aut(\frg)$. (Indeed $\Ad(\phi)$ is then an inner automorphism of $\frg$.) Hence, by Theorem~\ref{thm:compliant-involution}, $\alpha$ is compliant. This is consistent with Proposition~\ref{prop:even-normals-involution}, which says that $\phi \in \wG$ is sufficient to imply compliancy. However, compliancy is \emph{not} equivalent to the condition that $\phi \in \wG$, as the next example shows.

\begin{ex} \label{ex:involution-in-wG-not-sufficient}
We construct a calibration $\alpha \in \Lambda^3(\R^7)$ as follows. First, consider
$$ \beta = dx^1 \w dx^2 \w dx^3 + dx^4 \w dx^5 \w dx^6  \quad \in \Omega^3(\R^6), $$ 
which is a \emph{double point calibration}. That is, it has two calibrated $3$-planes (see \cite[Theorem 8]{DH}), namely $\frac{\partial}{\partial x^1} \w \frac{\partial}{\partial x^2} \w \frac{\partial}{\partial x^3}$ and $\frac{\partial}{\partial x^4} \w \frac{\partial}{\partial x^5} \w \frac{\partial}{\partial x^6}$. Let $\pi \colon \R^7 \to \R^6$ be the orthogonal projection  and define $\alpha = \pi^*\beta$. We claim that $\alpha$ and $\beta$ have the same calibrated planes. This follows from the fact that given unit-length vectors $v_j$ such that $u_j=\pi(v_j) \neq 0$, we have
$$ \alpha(v_1,v_2,v_3) = \beta(u_1,u_2,u_3) = |u_1||u_2||u_3| \, \beta\left(\frac{u_1}{|u_1|},\frac{u_2}{|u_2|},\frac{u_3}{|u_3|} \right)\leq |u_1||u_2||u_3| \leq 1. $$

Elements of $\rG$ should map calibrated $3$-planes into calibrated $3$-planes. The subgroup of $\SO{7}$ mapping each calibrated plane into itself is 
$\rH=\SO{3} \times \SO{3} \subset \rG$.
In addition, an element of $\SO{7}$ permuting the calibrated planes must be of the form
\begin{equation}
P = \begin{pmatrix} 0& P_2 & 0 \\ P_1 & 0 & 0 \\ 0 & 0 & \sigma \end{pmatrix} \in \rG, \qquad \sigma = \pm 1,
\end{equation}
and thus satisfies
\begin{align*}
P^*(dx^1 \w dx^2 \w dx^3) & = \det(P_1) dx^4 \w dx^5 \w dx^6, \\
 P^*(dx^4 \w dx^5 \w dx^6) & = \det(P_2) dx^1 \w dx^2 \w dx^3, \\
\det(P) & = - \sigma \det(P_1) \det(P_2).
\end{align*}
Therefore, $P \in \rG$ if and only if $P_1,P_2 \in \SO{3}$ and $\sigma = -1$. Let $M$ be the subset of matrices satisfying these conditions, then $\rG = (\SO{3} \times \SO{3}) \cup M$. In particular $\rG$ acts transitively on $\mathrm{Gr}_\alpha$.

Since $\rG$ is disconnected and the connected component of the identity is $\rH$ we have $\frg= \frh$. Therefore, by Definition~\ref{defn:compliant}, $\alpha$ is compliant. For the calibrated $3$-plane $W = \frac{\partial}{\partial x^1} \w \frac{\partial}{\partial x^2} \w \frac{\partial}{\partial x^3}$, we have
$$ \alpha \notin \bigoplus_{0 \leq 2r \leq 3} (\Lambda^{3-2r} W)  \otimes (\Lambda^{2r} W^{\perp}). $$
This shows, that having an even number of normal vectors which, by the proof of Proposition~\ref{prop:even-normals-involution}, is equivalent to $\phi \in \wG$, is in general a stronger condition than compliancy.
\end{ex}

Finally, we relate the notion of compliancy of $\alpha$ to the geometry of the calibrated Grassmanian $\mathrm{Gr}_{\alpha}$.

\begin{prop} \label{prop:symmetric-space}
As usual, assume that $\rG$ acts transitively on the $\alpha$-calibrated Grassmanian $\mathrm{Gr}_\alpha = \rG/\rH$. Suppose $\alpha$ is compliant. If $\rG$ is connected, then $\mathrm{Gr}_\alpha = \rG/\rH$ is locally symmetric. If, furthermore, $\mathrm{Gr}_\alpha$ is simply-connected, then it is symmetric.
\end{prop}
\begin{proof}
The Riemannian metric on the homogeneous space $\mathrm{Gr}_\alpha=\rG/\rH$ is induced from the bi-invariant metric on $\rG \subseteq \SO{n}$. Since this metric is $\Ad(\rG)$-invariant, for $X \in \frh$, $Y \in \frh^{\perp_g}$, and $P \in \rH$ we have $\Ad(P^{-1}) X \in \frh$, so
$$ \langle X, \Ad(P) Y \rangle = \langle \Ad(P^{-1}) X, Y \rangle = 0. $$
Thus $\frh^{\perp_g}$ is also $\Ad(H)$-invariant, and hence $\rG/\rH$ is \emph{reductive}, because $\frg =\frh \oplus \frh^{\perp_g}$.

The triple $(\frg,\frh, \Ad(\phi))$ is a symmetric Lie algebra in the sense of \cite[page 226]{KN} because $\alpha$ is compliant.
It follows from~\cite[Theorem 2.6 in Chapter X and Proposition 2.2 in Chapter XI]{KN} that the canonical connection on $\rG/\rH$ coincides with its Levi-Civita connection, and that the curvature is parallel.

Therefore,~\cite[Chapter XI, Theorem 6.2]{KN} shows that $\rG/\rH$ is locally symmetric. In addition, $\rG/\rH$ is complete because $\rG$ is compact. Consequently, \cite[Chapter XI, Theorems 6.2 and 6.3]{KN} imply that $\rG/\rH$ is a locally symmetric space, and it is a symmetric space if it is simply-connected.
\end{proof}

\begin{rmk} \label{rmk:Gr-locally-symmetric}
The hypothesis that $\rG$ is connected is taken throughout~\cite[Chapters X and XI]{KN}, and always holds for the interesting geometric examples. We are unaware if the converse of Proposition~\ref{prop:symmetric-space} holds. That is, it is unclear if $\mathrm{Gr}_{\alpha}$ being locally symmetric is equivalent to compliancy.
\end{rmk}

\begin{rmk} \label{rmk:Gr-simply-connected}
If $\rG$ is simply-connected and $\rH$ is connected, then the long exact sequence of homotopy groups for the fibration $\rH \hookrightarrow \rG \twoheadrightarrow \rG/\rH$ shows that $\mathrm{Gr}_\alpha$ is simply-connected. This condition holds in all the cases we are interested in, namely those in Table~\ref{table:GH}. It is of course well-known that all of these $\alpha$-calibrated Grassmanians $\mathrm{Gr}_{\alpha} = \rG/\rH$ are symmetric spaces.
\end{rmk}

\subsection{Compliancy of the important geometric calibrations} \label{sec:compliancy-examples}

We verify that the interesting geometric calibrations, namely those of Table~\ref{table:GH}, are all \emph{compliant}. By Proposition~\ref{prop:compliancy-well-defined}, it is enough to show compliancy with respect to a particular calibrated plane $W$. In each case we do this by showing that Proposition~\ref{prop:even-normals} applies. For illustration, we also show directly that the involution $\phi \in \OO{n}$ actually lies in $\wG$ in each case, as it must by the proof of Proposition~\ref{prop:even-normals-involution}.

In the following examples, $W$ is an $\alpha$-calibrated plane, with orthogonal complement $W^{\perp}$. Let $e_1, \ldots, e_k$ be an oriented orthonormal basis for $W$ and $\nu_1, \ldots, \nu_{n-k}$ be an oriented orthonormal basis for $W^{\perp}$.

\begin{ex} \label{ex:Kahler-compliant}
Consider the K\"ahler calibration $\alpha = \frac{1}{p!} \omega^p$ of Example~\ref{ex:Kahler}. The $2p$-plane $W$ spanned by $e_{2j-1} = dx^j, e^{2j} = dy^j$, for $1 \leq j \leq p$ is $\alpha$-calibrated, with $W^{\perp}$ spanned by $\nu^{2j-1} = dx^{p+j}, \nu^{2j} = dy^{p+j}$, for $1 \leq j \leq m-p$. Thus we have
\begin{equation} \label{eq:complex-forms}
\omega = \sum_{j=1}^p e^{2j-1} \w e^{2j} + \sum_{j=1}^{m-p} \nu^{2j-1} \w \nu^{2j} \, \in \Lambda^2 W \oplus \Lambda^2 W^{\perp}.
\end{equation}
It follows that $\alpha = \frac{1}{p!} \omega^p \in \oplus_{j = 0}^p (\Lambda^{2p-2j} W) \otimes (\Lambda^{2j} W^{\perp})$, so $\alpha$ is compliant. In this case it is clear that $\phi^* \omega = \omega$, so $\phi \in \wG$.
\end{ex}

\begin{ex} \label{ex:slag-compliant}
Consider the special Lagrangian calibration $\alpha = \real(\Upsilon)$ of Example~\ref{ex:slag}. The $m$-plane $W$ spanned by $e^j = dx^j$, for $1 \leq j \leq m$ is $\alpha$-calibrated, with $W^{\perp}$ spanned by $\nu^j = dy^j$, for $1 \leq j \leq m$. Thus we have
\begin{equation} \label{eq:SLAG-forms}
\Upsilon = (e^1 + i \nu^1) \w \cdots \w (e^m + i \nu^m), \qquad \omega = \sum_{j=1}^m e^j \w \nu^j.
\end{equation}
Then $\real(\Upsilon)$ consists of terms that involve an even number of the $\nu^j$, and $\imag(\Upsilon)$ consists of terms that involve an odd number of the $\nu^j$. Moreover, $\omega$ is purely type $(1,1)$. Thus Proposition~\ref{prop:even-normals} applies, so $\alpha$ is compliant. In this case, it is clear that $\phi$ corresponds to complex conjugation, so $\phi^* \Upsilon = \ol{\Upsilon}$. Thus $\phi^* \omega = \omega$, $\phi^* \real(\Upsilon) = \real(\Upsilon)$, and $\phi^* \imag(\Upsilon) = - \imag(\Upsilon)$, so $\phi \in \wG$. This example shows why we needed to define $\wG$ in~\eqref{eq:hatG-defn} by $\phi^* \alpha_j = \pm \alpha_j$, rather than just $\phi^* \alpha_j = \alpha_j$. Note also that $\det \phi = (-1)^m$, so it is not always true that $\phi \in \SO{2m}$. Thus this example also shows why we needed to define $\wG$ in~\eqref{eq:hatG-defn} as a subgroup of $\OO{n}$ rather than $\SO{n}$.
\end{ex}

\begin{ex} \label{ex:assoc-compliant}
Consider the associative calibration $\alpha = \ph$ of Example~\ref{ex:assoc}. The $3$-plane $W$ spanned by $e^j = dx^j$, for $1 \leq j \leq 3$ is $\alpha$-calibrated, with $W^{\perp}$ spanned by $\nu^j = dx^{3+j}$, for $1 \leq j \leq 4$. We see from~\eqref{eq:assoc-form} that $\alpha$ is of type $(3,0) + (1,2)$. Thus $\alpha$ is compliant. It is clear that $\phi^* \ph = \ph$, so $\phi \in \G \subset \wG$.
\end{ex}

\begin{ex} \label{ex:coassoc-compliant}
Consider the coassociative calibration $\alpha = \ps$ of Example~\ref{ex:coassoc}. The $4$-plane $W$ spanned by $e^j = dx^{3+j}$, for $1 \leq j \leq 4$ is $\alpha$-calibrated, with $W^{\perp}$ spanned by $\nu^j = dx^j$, for $1 \leq j \leq 3$. We see from~\eqref{eq:coassoc-form} that $\alpha$ is of type $(4,0) + (2,2)$. Thus $\alpha$ is compliant. It is clear that $\phi^* \ps = \ps$, but $\phi^* \ph = - \ph$, so $\phi \in \wG$ although $\phi$ is not in $\SO{7}$.
\end{ex}

\begin{ex} \label{ex:Cayley-compliant}
Consider the Cayley calibration $\alpha = \Ph$ of Example~\ref{ex:Cayley}. The $4$-plane $W$ spanned by $e^j = dx^j$, for $1 \leq j \leq 4$ is $\alpha$-calibrated, with $W^{\perp}$ spanned by $\nu^j = dx^{4+j}$, for $1 \leq j \leq 4$. We see from~\eqref{eq:Cayley-form} that $\alpha$ is of type $(4,0) + (2,2) + (0,4)$. Thus $\alpha$ is compliant. It is clear that $\phi^* \Ph = \Ph$, so $\Ph \in \Spin{7} \subset \wG$.
\end{ex}

\section{Extrinsic geometry of calibrated submanifolds} \label{sec:extrinsic-calibrated}

We review the classical Gauss--Codazzi--Ricci relations between intrinsic and extrinsic geometry of a Riemannian immersion, from a more modern point of view than that which is traditionally presented. This serves to both set notation and to keep the paper reasonably self-contained. We then discuss calibrated submanifolds, and prove our main theorem which characterizes the extrinsic geometry of a calibrated submanifold, when the calibration is compliant. We examine this theorem for the interesting geometric examples, and use the result to motivate the definition of an \emph{infinitesimally calibrated} submanifold.

\subsection{Extrinsic geometry of Riemannian immersions} \label{sec:extrinsic-review}

Let $(M^n, g)$ be a Riemannian manifold. Let $\ol{\nabla}$ denote the Levi-Civita connection on $TM$ induced by $g$. Let $\iota \colon L \to M$ be an immersion of a $k$-dimensional manifold $L$ into $M$. The results in this paper are all local, so we can safely think of $L$ as an embedded submanifold and $\iota$ as the inclusion. Equipping $L$ with the pullback (restricted) metric $\iota^* g = \rest{g}{L}$, the map $\iota \colon (L, \rest{g}{L}) \to (M, g)$ becomes a \emph{Riemannian} (or \emph{isometric}) immersion.

The restriction of $TM$ to $L$, which is the pullback bundle by the map $\iota \colon L \to M$, will be denoted
$$ E : = \rest{TM}{L} = \iota^* (TM). $$
Because $\iota \colon L \to M$ is an immersion, it induces a bundle map $\iota_* \colon TL \to E$ which exhibits $TL$ as a smooth subbundle of $E$. Explicitly, $T_x L$ is identified with its image $\iota_* (T_x L) \subseteq T_x M$. Note that $g$ is a fibre metric on $TM$, so $\iota^* g$ is a fibre metric on $E = \rest{TM}{L}$. We therefore have an orthogonal direct sum decomposition
\begin{equation} \label{eq:splitting}
E = \rest{TM}{L} = TL \oplus NL,
\end{equation}
where $TL$ is the tangent bundle of $L$ and $NL$ is the \emph{normal} bundle of $L$, defined to be the orthogonal complement of $TL$ in $E$ with respect to the fibre metric $\iota^* g$. The vector bundles $TL$ and $NL$ are thus both equipped with fibre metrics, and are of ranks $k$ and $n-k$, respectively.

We use angle brackets $\langle \cdot, \cdot \rangle$ to denote the fibre metrics on $E$, $TL$, and $NL$. There should be no possibility of confusion. We also use $u, v, w, y$ to denote smooth sections of $TL$, which are tangent vector fields to $L$, and we use $\xi, \eta, \zeta$ to denote smooth sections of $NL$, which are \emph{normal} vector fields to $L$. Finally, we use $\{ e_1, \ldots, e_k \}$ to denote a (local) orthonormal frame for $TL$. 

By an abuse of notation, we use $\ol{\nabla}$ to also denote the pullback $\iota^* \ol{\nabla}$ of the connection $\ol{\nabla}$ from $TM$ to $\iota^* (TM) = E = TL \oplus NL$. Let $\pi_{TL} \colon E \to TL$ and $\pi_{NL} \colon E \to NL$ be the orthogonal projection bundle maps from $E$ onto $TL$ and $NL$, respectively. Then $\nabla = \pi_{TL} \circ \ol{\nabla}$ is a connection on $TL$, and $D = \pi_{NL} \circ \ol{\nabla}$ is a connection on $NL$. It is a standard fact (and is easy to check) that $\nabla$ is the Levi-Civita connection of $(L, \rest{g}{L})$, and $D$ is called the \emph{normal connection} of the immersion. Note that $\nabla$ and $D$ are compatible with the fibre metrics on $TL$ and $NL$, respectively.

For $u, v \in \Gamma(TL)$ and $\xi \in \Gamma(NL)$, we define
\begin{align*}
A(u, v) & := \pi_{NL} (\ol{\nabla}_u v) = \ol{\nabla}_u v - \nabla_u v, \\
-A^{\xi} u & := \pi_{TL} (\ol{\nabla}_u \xi) = \ol{\nabla}_u \xi - D_u \xi.
\end{align*}
It is easy to see that $A(u,v)$ and $A^{\xi} u$ are $C^{\infty}(L)$-linear in $u, v, \xi$. By definition, we have
\begin{equation} \label{eq:2ndff-defn}
\begin{aligned}
\ol{\nabla}_u v & = \pi_{TL} (\ol{\nabla}_u v) + \pi_{NL} (\ol{\nabla}_u v) = \nabla_u v + A(u,v), \\
\ol{\nabla}_u \xi & = \pi_{TL}  (\ol{\nabla}_u \xi) + \pi_{NL} (\ol{\nabla}_u \xi) = - A^{\xi} u + D_u \xi.
\end{aligned}
\end{equation}

Since $\ol{\nabla}$ and $\nabla$ are both Levi-Civita connections, hence torsion-free, we have
$$ A(u, v) - A(v,u) = (\ol{\nabla}_u v - \ol{\nabla}_v u) - (\nabla_u v - \nabla_v u) = [u. v] - [u, v] = 0. $$
Thus $A(\cdot, \cdot)$ is a symmetric $NL$-valued bllinear form on $L$, called the \emph{second fundamental form}.

From $\langle v, \xi \rangle = 0$, metric compatibility, and~\eqref{eq:2ndff-defn}, we get
\begin{equation} \label{eq:2ndff-temp}
0 = u \langle v, \xi \rangle = \langle \ol{\nabla}_u v, \xi \rangle + \langle v, \ol{\nabla}_u \xi \rangle = \langle A(u, v), \xi \rangle - \langle v, A^\xi u \rangle.
\end{equation}
Symmetry of $A(\cdot, \cdot)$ thus implies that the bundle map $A^{\xi} \colon TL \to TL$ is pointwise \emph{self-adjoint}. This map is called the \emph{shape operator} or \emph{Weingarten map} of the immersion.

If we define $A^{\xi}(u, v) = \langle A(u, v), \xi \rangle$, then by~\eqref{eq:2ndff-temp} we have
\begin{equation} \label{eq:2ndff}
A^{\xi} (u,v) = \langle A(u, v), \xi \rangle = \langle A^{\xi} u, v \rangle = \langle u, A^{\xi} v \rangle = \langle A(v, u), \xi \rangle = A^{\xi} (v, u).
\end{equation}
The above shows explicitly that the second fundamental form and the shape operator are \emph{equivalent} pieces of data about the immersion, encoding the \emph{extrinsic geometry}.

For later use, we recall the following definition. The \emph{mean curvature} $H$ of $L$ in $M$ is the trace of the second fundamental form. That is,
$$ H = \sum_{j=1}^k A(e_j, e_j). $$
Note that $H$ is a \emph{normal} vector field on $L$. That is, it is a section of $NL$.

It is useful to modify our notation as follows. For $u \in \Gamma(TL)$, define $A_u \colon \Gamma(TL) \to \Gamma(NL)$ by
\begin{equation} \label{eq:A-modified}
A_u v := A(u,v),
\end{equation}
so that by~\eqref{eq:2ndff} we have $\langle A_u v, \xi \rangle = \langle v, A^{\xi} u \rangle$. This means that the adjoint $A_u^* \colon \Gamma(NL) \to \Gamma(TL)$ is given by
\begin{equation} \label{eq:A-star}
A_u^* \xi := A^{\xi} u.
\end{equation}
With this notation, the pair of equations~\eqref{eq:2ndff-defn} become
\begin{equation} \label{eq:2ndff-defn-b}
\begin{aligned}
\ol{\nabla}_u v & = (\ol{\nabla}_u v)^T + (\ol{\nabla}_u v)^{\perp} = \nabla_u v + A_u v, \\
\ol{\nabla}_u \xi & = (\ol{\nabla}_u \xi)^T + (\ol{\nabla}_u \xi)^{\perp} = - A_u^* \xi + D_u \xi.
\end{aligned}
\end{equation}
Thus, if $(v, \xi)$ is a section of $TL \oplus NL$, and $u$ is a vector field on $L$, then by~\eqref{eq:2ndff-defn-b} we have
\begin{equation} \label{eq:nablabar}
\ol{\nabla}_u (v, \xi) = \ol{\nabla}_u (v, 0) + \ol{\nabla}_u (0, \xi) = (\nabla_u v - A_u^* \xi, D_u \xi + A_u v).
\end{equation}
Observe that $\wh{\nabla} = \nabla \oplus D$ is also a connection on $TL \oplus NL$, where
\begin{equation} \label{eq:nablahat}
\wh{\nabla}_u (v, \xi) = (\nabla_u v, D_u \xi).
\end{equation}
If we write an element $(v, \xi)$ of the direct sum $\Gamma(TL \oplus NL) = \Gamma(TL) \oplus \Gamma(NL)$ as the $2 \times 1$ matrix $(v, \xi)^t$, then equation~\eqref{eq:nablahat} says that $\wh{\nabla}_u$ corresponds to the block diagonal matrix
\begin{equation} \label{eq:nablahat-2}
\wh{\nabla}_u = \begin{pmatrix} \nabla_u & 0 \\ 0 & D_u \end{pmatrix}.
\end{equation}
Define a global smooth section $B$ of the bundle $T^*L \otimes \End(TL \oplus NL)$ as follows. For $u \in \Gamma(TL)$, we set $B_u : = B(u)$ to be the section of $\End (TL \oplus NL)$ given by
\begin{equation} \label{eq:B}
B_u (v, \xi) = (- A_u^* \xi, A_u v).
\end{equation}
Similarly, as an operator on the direct sum $\Gamma(TL \oplus NL) = \Gamma(TL) \oplus \Gamma(NL)$, equation~\eqref{eq:B} says that $B_u$ corresponds to the skew-adjoint matrix
\begin{equation} \label{eq:B-2}
B_u = \begin{pmatrix} 0 & - A_u^* \\ A_u & 0 \end{pmatrix}.
\end{equation}
It is clear that $B$ contains exactly the same information as the second fundamental form, or the shape operator. For this reason, we abuse notation and refer to $B$ as the second fundamental form of the immersion.

From equations~\eqref{eq:2ndff-defn-b},~\eqref{eq:nablabar},~\eqref{eq:nablahat}, and~\eqref{eq:B}, we get
\begin{equation} \label{eq:nabla-relation}
\ol{\nabla}_u = \wh{\nabla}_u + B_u.
\end{equation}
Note that~\eqref{eq:B-2} shows that $B_u \colon \Gamma(E) \to \Gamma(E)$ is \emph{skew-adjoint} for each $u \in \Gamma(TL)$, which is precisely what is expected from~\eqref{eq:nabla-relation}, since both $\ol{\nabla}$ and $\wh{\nabla}$ are metric-compatible connections on $E$.

Let $\ol{R}$ and $\wh{R}$ denote the curvatures of the connections $\ol{\nabla}$ and $\wh{\nabla}$, respectively, and let $s = (v, \xi)^t \in \Gamma(E) = \Gamma(TL \oplus NL)$. Using~\eqref{eq:nabla-relation} and the Leibniz rule $\wh{\nabla}_u (B_v s) = (\wh{\nabla}_u B_v) s + B_v (\wh{\nabla}_u s)$, we compute
\begin{align*}
\ol{R}_{u, v} s & = \ol{\nabla}_u (\ol{\nabla}_v s) - \ol{\nabla}_v (\ol{\nabla}_u s) - \ol{\nabla}_{[u,v]} s \\
& = \ol{\nabla}_u (\wh{\nabla}_v s + B_v s) - \ol{\nabla}_v (\wh{\nabla}_u s + B_u s) - \wh{\nabla}_{[u,v]} s - B_{[u,v]} s \\
& = \wh{\nabla}_u (\wh{\nabla}_v s) + \wh{\nabla}_u (B_v s) + B_u (\wh{\nabla}_v s) + B_u (B_v s) \\
& \qquad {} - \wh{\nabla}_v (\wh{\nabla}_u s) - \wh{\nabla}_v (B_u s) - B_v (\wh{\nabla}_u s) - B_v (B_u s) - \wh{\nabla}_{[u,v]} s - B_{[u,v]} s \\
& = \wh{R}_{u,v} s + (\wh{\nabla}_u B_v) s - (\wh{\nabla}_v B_u) s + (B_u B_v - B_v B_u) s - B_{[u,v]} s.
\end{align*}
Observing that $\wh{\nabla}_u v = \nabla_u v$ and $[u, v] = \nabla_u v - \nabla_v u$, we have
\begin{align*}
(\wh{\nabla}_u B_v) - (\wh{\nabla}_v B_u) - B_{[u,v]} & = (\wh{\nabla}_u B)_v + B_{\wh{\nabla}_u v} - (\wh{\nabla}_v B)_u - B_{\wh{\nabla}_v u} - B_{(\nabla_u v - \nabla_v u)} \\
& = (\wh{\nabla}_u B)_v - (\wh{\nabla}_v B)_u.
\end{align*}
Combining these computations yields the fundamental relation between $\ol{R}$, $\wh{R}$, and $\wh{\nabla} B$, which is
\begin{equation} \label{eq:GCR-general}
\ol{R}_{u,v} = \wh{R}_{u,v} + (\wh{\nabla}_u B)_v - (\wh{\nabla}_v B)_u + [B_u, B_v],
\end{equation}
where $[B_u, B_v] = B_u B_v - B_v B_u$ is the commutator of linear operators.

\begin{rmk} \label{rmk:GCR}
Another common way of expressing~\eqref{eq:GCR-general} is in terms of the exterior covariant derivative $d^{\wh{\nabla}}$ of the connection $\wh{\nabla}$. Since $\ol{\nabla} = \wh{\nabla} + B$, it is well-known that
$$ \ol{R} = \wh{R} + d^{\wh{\nabla}} B + B \w B, $$
which is equivalent to~\eqref{eq:GCR-general}.
\end{rmk}

The fundamental relation~\eqref{eq:GCR-general} is a reformulation of the classical Gauss--Codazzi--Ricci equations of Riemannian submanifold theory. To see this, first observe from~\eqref{eq:nablahat-2} and~\eqref{eq:B-2} that
\begin{equation} \label{eq:GCR-temp}
\begin{aligned}
\wh{R}_{u,v} & = \begin{pmatrix} R^{\nabla}_{u,v} & 0 \\ 0 & R^D_{u, v} \end{pmatrix}, \\
(\wh{\nabla}_u B)_v - (\wh{\nabla}_v B)_u & = \begin{pmatrix} 0 & - (\wh{\nabla}_u A)^*_v + (\wh{\nabla}_v A)^*_u \\ (\wh{\nabla}_u A)_v - (\wh{\nabla}_v A)_u & 0 \end{pmatrix}, \\
B_u B_v - B_v B_u & = \begin{pmatrix} - A_u^* A_v + A_v^* A_u & 0 \\ 0 & - A_u A_v^* + A_v A_u^* \end{pmatrix}.
\end{aligned}
\end{equation}
As an aside, note that the commutator $[B_u, B_v]$ of the two skew-adjoint endomorphism $B_u, B_v$ is itself skew-adjoint, as it must be. Moreover, both operators $R^{\nabla}_{u,v}$ and $R^D_{u,v}$ are skew-adjoint, since $\nabla$ and $D$ are metric-compatible connections, as is their direct sum $\ol{R}_{u,v} = R^{\nabla}_{u,v} \oplus R^D_{u,v}$. Thus both sides of all three equations above are skew-adjoint endomorphisms. Expressing~\eqref{eq:GCR-general} in terms of the above $2 \times 2$ matrices, we have
\begin{equation} \label{eq:GCR-matrix-form}
\ol{R}_{u,v} = \begin{pmatrix} R^{\nabla}_{u,v} - A_u^* A_v + A_v^* A_u & - (\wh{\nabla}_u A)^*_v + (\wh{\nabla}_v A)^*_u \\ (\wh{\nabla}_u A)_v - (\wh{\nabla}_v A)_u & R^D_{u, v} - A_u A_v^* + A_v A_u^* \end{pmatrix}.
\end{equation}
The upper left block of~\eqref{eq:GCR-matrix-form} gives the classical \emph{Gauss equation}, which is
\begin{align*}
\langle \ol{R}_{u,v} w, y \rangle & = \langle (R^{\nabla}_{u,v} - A_u^* A_v + A_v^* A_u) w, y \rangle \\
& = \langle R^{\nabla}_{u,v} w, y \rangle + \langle A_u w, A_v y \rangle - \langle A_v w, A_u y \rangle.
\end{align*}
The lower right block of~\eqref{eq:GCR-matrix-form} gives the classical \emph{Ricci equation}, which, using~\eqref{eq:A-star} and the fact that the shape operator $A^{\xi}$ is self-adjoint, is
\begin{align*}
\langle \ol{R}_{u,v} \xi, \eta \rangle & = \langle (R^D_{u,v} - A_u A_v^* + A_v A_u^*) \xi, \eta \rangle \\
& = \langle R^D_{u,v} \xi, \eta \rangle + \langle A_u^* \xi, A_v^* \eta \rangle - \langle A_v^* \xi, A_u^* \eta \rangle \\
& = \langle R^D_{u,v} \xi, \eta \rangle + \langle A^{\xi} u, A^{\eta} v \rangle - \langle A^{\xi} v, A^{\eta} u \rangle \\
& = \langle R^D_{u,v} \xi, \eta \rangle - \langle (A^{\xi} A^{\eta} - A^{\eta} A^{\xi}) u, v \rangle.
\end{align*}
Finally, either of the off diagonal blocks of~\eqref{eq:GCR-matrix-form} gives the classical \emph{Codazzi equation}, which, using~\eqref{eq:A-modified} and~\eqref{eq:2ndff}, is
\begin{align*}
\langle \ol{R}_{u,v} w, \xi \rangle & = \langle ((\wh{\nabla}_u A)_v - (\wh{\nabla}_v A)_u) w, \xi \rangle \\
& = \langle (\wh{\nabla}_u A)(v,w) - (\wh{\nabla}_v A)(u, w), \xi \rangle \\
& = (\nabla_u A)^{\xi} (v, w) - (\nabla_v A)^{\xi} (u, w).
\end{align*}

\subsection{Calibrated submanifolds} \label{sec:calibrated}

Let $\alpha$ be a \emph{calibration} $k$-form on $(M, g)$. This means $\alpha$ is a \emph{closed} $k$-form, and that $\rest{\alpha}{x}$ is a calibration on the inner product space $(T_x M, g_x)$ for all $x \in M$. That is, $\rest{\alpha}{W_x} \leq \vol_{W_x}$ for any oriented $k$-plane $W_x \subset T_x M$, for any $x \in M$. An oriented $k$-plane $W_x \subset T_x M$ is called $\alpha$-\emph{calibrated} if $\rest{\alpha}{W_x} = \vol_{W_x}$. Let $L$ be an \emph{oriented} $k$-dimensional submanifold of $M$. We use $\vol_{L}$ to denote the induced volume form on $L$ from the given orientation and the induced metric $\rest{g}{L}$. We say that $L$ is $\alpha$-\emph{calibrated} if $T_x L$ is calibrated with respect to $\rest{\alpha}{x}$ for all $x \in L$. Equivalently, $L$ is $\alpha$-calibrated if
\begin{equation} \label{eq:calibrated}
\rest{\alpha}{L} = \vol_L.
\end{equation}
The fundamental theorem of calibrated geometry~\cite{HL} says that a calibrated submanifold is homologically volume minimizing, and thus in particular must have vanishing mean curvature $H$.

The condition~\eqref{eq:calibrated} that $L$ is calibrated is a \emph{first order} condition on the immersion $\iota \colon L \to M$, as it involves the pullback $\rest{\alpha}{L} = \iota^* \alpha$. By contrast, the extrinsic geometry of the Riemannian immersion $\iota \colon L \to M$, given by the second fundamental form $A$ and the induced tangent and normal connections $\nabla$ on $TL$ and $D$ on $NL$, respectively, is \emph{second order} information, as their definitions involve differentiation of sections of $\rest{TM}{L}$. In Section~\ref{sec:main-thm} we characterize the conditions imposed on the extrinsic geometric data $(A, \nabla, D)$ when a Riemannian immersion $\iota \colon L \to M$ is calibrated with respect to a \emph{parallel} and \emph{compliant} calibration.

\begin{defn} \label{defn:compliant-on-manifolds}
Suppose that $(M,g)$ is equipped with a $\rG$-structure (a reduction of the structure group of the frame bundle from $\GL{n, \R}$ to $\rG$) and that $\rG$ is of the form~\eqref{eq:G-defn}. This means there exist differential forms $\alpha_0, \ldots, \alpha_N$ on $M$, where $\alpha = \alpha_0$ is a calibration, such that for all $x \in M$ we have
$$ \rG_x = \{ P \in \SO{T_x M} : P^* \rest{\alpha_j}{x} = \rest{\alpha_j}{x}, \, \text{for $0 \leq j \leq N$} \} $$
corresponds to the conjugacy class of $\rG \subset \SO{n}$ under any oriented isometry $(T_x M, g_x) \cong (\R^n, \langle \cdot, \cdot \rangle)$.

Suppose further that the group $\rG_x$ acts transitively on the $\alpha_x$-calibrated Grassmanian $\mathrm{Gr}_{\alpha_x}$, so that $\mathrm{Gr}_{\alpha_x} = \rG_x / \rH_x$, where $\rH_x = \{ P \in \rG_x : P(W) = W \}$ for any $W \in \mathrm{Gr}_{\alpha_x}$.
Then we say $\alpha$ is \emph{compliant} if $\alpha_x$ is a compliant calibration on $(T_x M, g_x)$ in the sense of Definition~\ref{defn:compliant}. (Note that this condition holds for some $x \in M$ if and only if it holds for all $x \in M$.)
\end{defn}

\begin{defn} \label{defn:parallel-compliant}
Let $\alpha$ be a compliant calibration on $(M,g)$ in the sense of Definition~\ref{defn:compliant-on-manifolds}. Thus we have a finite set $\{ \alpha_0 = \alpha, \alpha_1, \ldots, \alpha_N \}$ of differential forms which are stabilized by $\rG$. We say $\alpha$ is \emph{parallel-compliant} if $\ol{\nabla} \alpha_j = 0$ for all $0 \leq j \leq N$. In particular, this means that $\alpha$ itself is parallel, but so are the rest of the forms $\alpha_1, \ldots, \alpha_N$. Moreover, from~\eqref{eq:G-defn} we obtain that $\Hol^0(\ol{\nabla}) \subset \rG$ and $\hol(\ol{\nabla}) \subset \frg$.
\end{defn}

\begin{rmk} \label{rmk:parallel-compliant}
In all the interesting geometric cases of interest, from Table~\ref{table:GH}, the calibration $\alpha$ being parallel-compliant in the sense of Definition~\ref{defn:parallel-compliant} is equivalent to the $\rG$-structure being \emph{torsion-free}.
\end{rmk}

Recall the orthogonal decomposition $E = \rest{TM}{L} = TL \oplus NL$ of~\eqref{eq:splitting}. We similarly have
$$ E^* = \rest{T^*M}{L} = T^*L \oplus N^*L, $$
where $T^*L$ is the dual bundle of $TL$, the \emph{cotangent bundle} of $L$, and $N^*L$ is the dual bundle of $NL$, the \emph{conormal bundle} of $L$. Note that $N^*L$ is the annihilator of $TL$ and similarly $T^*L$ is the annihilator of $NL$. In particular, we can pullback the bundle $\Lambda^k T^* M$ over $M$ by the smooth map $\iota \colon L \to M$ to obtain
$$ \iota^* (\Lambda^k T^* M) = \Lambda^k (\iota^* T^* M) = \Lambda^k (T^*L \oplus N^*L) = \bigoplus_{p+q=k} (\Lambda^p T^*L) \otimes (\Lambda^q N^*L). $$
The space $\Gamma( \Lambda^k E^*)$ of smooth sections thus decomposes as
\begin{equation*}
\Gamma( \Lambda^k E^*) = \bigoplus_{p+q=k} \Gamma \big( (\Lambda^p T^*L) \otimes (\Lambda^q N^*L) \big).
\end{equation*}
We say that an element of $\Gamma \big( (\Lambda^p T^*L) \otimes (\Lambda^q N^*L) \big)$ is of type $(p,q)$.

Let $\beta$ be a $k$-form on $M$. The pullback of $\beta \in \Gamma( \Lambda^k T^* M )$ by $\iota \colon L \to M$, \emph{as a section of a vector bundle, not as a differential form}, is a section $\iota^* \beta \in \Gamma( \iota^* \Lambda^k T^* M )$, and thus we can write
\begin{equation} \label{eq:forms-pq}
\iota^* \beta = \sum_{p+q = k} \beta^{p,q} \qquad \text{ where $\beta^{p,q}$ is of type $(p, q)$.}
\end{equation}
Note that $\beta^{k,0} \in \Gamma( \Lambda^k T^*L)$ is a smooth $k$-form on $L$, and is exactly the pullback of $\beta$ by $\iota$ \emph{as a differential form}. In particular, if $\alpha$ is a calibration $k$-form then $L$ is $\alpha$-calibrated if and only if $\alpha^{k,0} = \vol_L$.

\subsection{The main theorem: extrinsic geometry of calibrated submanifolds} \label{sec:main-thm}

In this section we characterize the conditions imposed on the extrinsic geometric data $(A, \nabla, D)$ for a Riemannian immersion $\iota \colon L \to M$ which is $\alpha$-calibrated for a \emph{parallel-compliant} calibration $\alpha$ on $(M, g)$.

Let $\alpha$ be a compliant calibration on $(M^n, g)$, as given by Definition~\ref{defn:compliant-on-manifolds}, and suppose that $L$ is $\alpha$-calibrated. Thus with
$$ \rH_x = \rG_x \cap (\SO{T_xL} \times \SO{N_xL}), $$
and hence
\begin{equation*}
\frh_x = \frg_x \cap (\Lambda^2 T_xN \oplus \Lambda^2N_xL),
\end{equation*}
the compliancy assumption says that
\begin{equation*}
\frh_x^{\perp_{\frg_x}} = \frg_x \cap (T_xN \otimes N_xL).
\end{equation*}

\begin{thm} \label{thm:main}
Let $\alpha$ be a parallel-compliant calibration, and let $L$ be an $\alpha$-calibrated submanifold. For $u \in \Gamma(TL)$, let $\wh{\nabla}_u$ and $B_u$ be as given by~\eqref{eq:nablahat} and~\eqref{eq:B}, respectively. Then we have
\begin{equation} \label{eq:main}
\wh{\nabla} \alpha_j = 0 \quad \text{and} \quad B_u \diamond \alpha_j = 0 \quad \text{for all $0 \leq j \leq N$}.
\end{equation}
Consequently, $\Hol^0 (\wh{\nabla}) \in \rH$ (so $\hol(\wh{\nabla}) \in \frh$), and $B_u \in \frg$.
\end{thm}
\begin{proof}
The bundle $TM$ has a $\rG$-structure, and the bundle $E = \rest{TM}{L} = TL \oplus NL$ has an $\rH$-structure, and $\rH \subset \rG$. This means that in a neighbourhood of any point in $L$, we can find a local frame of $E$ of the form $\{ e_1, \ldots, e_k, \nu_1, \ldots, \nu_{n-k} \}$, where $\{ e_1, \ldots, e_k \}$ is a local oriented orthonormal frame for $TL$, and $\{ \nu_1, \ldots, \nu_{n-k} \}$ is a local oriented orthonormal frame for $NL$, with respect to which the forms $\alpha = \alpha_0, \alpha_1, \ldots, \alpha_N$ have constant coefficients. (In such a frame they agree with the standard models in Euclidean space.)

To simplify notation, use $\beta$ to denote an $m$-form $\alpha_j$ for some $0 \leq j \leq N$, so $\ol{\nabla} \beta = 0$. Using the Leibniz rules for $\ol{\nabla}$ and $\wh{\nabla}$, as well as the relation $\ol{\nabla}_u = \wh{\nabla}_u + B_u$ from~\eqref{eq:nabla-relation}, we compute
\begin{align*}
0 & = (\ol{\nabla}_u \beta)(v_1, \ldots, v_m) \\
& = u \big( \beta(v_1, \ldots, v_m) \big) - \beta( \ol{\nabla}_u v_1, \ldots, v_m) - \cdots - \beta( v_1, \ldots, \ol{\nabla}_u v_m) \\
& = (\wh{\nabla}_u \beta) (v_1, \ldots, v_m) + \beta (\wh{\nabla}_u v_1, \ldots, v_m) + \cdots + \beta(v_1, \ldots, \wh{\nabla}_u v_m) \\
& \qquad {} - \beta( \ol{\nabla}_u v_1, \ldots, v_m) - \cdots - \beta( v_1, \ldots, \ol{\nabla}_u v_m) \\
& = (\wh{\nabla}_u \beta) (v_1, \ldots, v_m) - \beta( B_u v_1, \ldots, v_m ) - \cdots - \beta( v_1, \ldots, B_u v_m) \\
& = \big( \wh{\nabla}_u \beta - B_u \diamond \beta \big) (v_1, \ldots, v_m).
\end{align*}
We thus have
\begin{equation} \label{eq:main-temp}
\wh{\nabla}_u \beta = B_u \diamond \beta.
\end{equation}
(Note that this part only uses the fact that $\ol{\nabla}{\beta} = 0$.)

Let $\wh{\Gamma}^c_{ab}$ denote the Christoffel symbols of $\wh{\nabla}$ with respect to our frame $\{ e_1, \ldots, e_k, \nu_1, \ldots, \nu_{n-k} \}$, which are local $\Lambda^2 E$-valued $1$-forms. Write $\wh{\Gamma}_u$ for the local $\End(E)$-valued section $\wh{\Gamma}(u)$.
Because $\beta$ has constant coefficients with respect to this frame, we have
$$ (\wh{\nabla}_u \beta)_{i_1 \cdots i_m} = - (\wh{\Gamma}_u)^c_{i_1} \beta_{c i_2 \cdots i_m} - \cdots - (\wh{\Gamma}_u)^c_{i_m} \beta_{i_1 \cdots i_{m-1} c} = - (\wh{\Gamma}_u \diamond \beta)_{i_1 \cdots i_m}, $$
so $\wh{\nabla}_u \beta = - \wh{\Gamma}_u \diamond \beta$, which allows us to write~\eqref{eq:main-temp} as
\begin{equation} \label{eq:main-temp2}
(\wh{\Gamma}_u + B_u ) \diamond \beta = 0,
\end{equation}
so for every $x \in L$, we have $\wh{\Gamma}_u + B_u \in \frg_x$. Because of the block diagonal form~\eqref{eq:nablahat-2} of $\wh{\nabla}_u = \nabla_u \oplus D_u$, we have that $\wh{\Gamma}_u$ is a local section of $\Lambda^2 TL \oplus \Lambda^2 NL$. Because of the block anti-diagonal form~\eqref{eq:B-2} of $B_u$, we have that $B_u$ is a section of $TL \otimes NL$.

Compliancy says that $ \frg_x = \frh_x \oplus \frh_x^{\perp_{\frg_x}}$, with $\frh_x = \frg_x \cap (\Lambda^2 T_xL \oplus \Lambda^2 N_xL)$ and $\frh_x^{\perp_{\frg_x}} = \frg_x \cap (T_x L \otimes N_x L)$. Thus for every $x \in L$ we have $\wh{\Gamma}_u \in \frh_x \subset \frg_x$ and $B_u \in \frh_x^{\perp_{\frg_x}} \subset \frg_x$. This implies that $\wh{\nabla}_u \beta = - \wh{\Gamma}_u \diamond \beta = 0$ and $B_u \diamond \beta = 0$.

The statement about the holonomy of $\wh{\nabla}$ follows from $\wh{\Gamma}_u \in \frh$.
\end{proof}

We make two remarks about the conditions which are consequences of Theorem~\ref{thm:main}.

\begin{rmk} \label{rmk:VCP}
Suppose that $\alpha$ is determined by a $(k-1)$-fold vector cross product $P \colon \Lambda^{k-1} TM \to TM$. That is,
$$ \alpha(X_1, \ldots, X_k) = g ( P(X_1, \ldots, X_{k-1}), X_k ). $$
(This is the case for $\alpha = \omega$, $\ph$, and $\Ph$ from Section~\ref{sec:compliancy-examples}.)

Then the condition $B_u \diamond \alpha = 0$ says that along $L$ we have
\begin{align*}
0 & = (B_u \diamond \alpha)(X_1, \ldots, X_k) \\
& = \alpha(B_u X_1, \ldots, X_k) + \cdots + \alpha(X_1, \ldots, B_u X_k) \\
& = g (P(B_u X_1, \ldots, X_{k-1}), X_k) + \cdots + g (P(X_1, \ldots, B_u X_{k-1}), X_k) + g(P(X_1, \ldots, X_{k-1}), B_u X_k).
\end{align*}
Using the fact that $B_u$ is skew-adjoint, this is equivalent to
\begin{equation} \label{eq:VCP-case}
B_u (P(X_1,\ldots, X_{k-1})) = P(B_u X_1, \ldots, X_{k-1}) + \cdots + P(X_1, \ldots, B_u X_{k-1}),
\end{equation}
which says that $P$ is in the kernel of the induced infinitesimal action of $B_u$.
\end{rmk}

\begin{rmk} \label{rmk:parallel-bundle-map}
Let $\alpha_j = \sum_{r=0}^k \alpha_j^{k-r,r}$ be the type $(p,q)$ decomposition of $\alpha_j$ defined in~\eqref{eq:forms-pq}. Since $\wh{\nabla}$ preserves types, the condition $\wh{\nabla} \alpha_j = 0$ implies that $\wh{\nabla} \alpha_j^{k-r,r} = 0$ for all $0 \leq r \leq k$. When $r=0$ and $j=0$, we know that $\alpha_0^{k,0} = \vol_L$, and $\wh{\nabla} = \nabla$ on $\Gamma(\Lambda^k T^* L)$, so this is automatic. But if $0 < r \leq k$ is such that $\alpha_j^{k-r,r} \neq 0$, then the condition $\wh{\nabla} \alpha_j^{k-r,r} = 0$ is nontrivial. This means we have a $\wh{\nabla}$-parallel bundle map $\alpha_j^{k-r,r} \colon (\Lambda^{k-r} T^*L, \nabla) \to (\Lambda^r N^* L, D)$. Precomposing this with the Hodge star operator of $L$, we obtain a parallel bundle map
$$ (\Lambda^r T^*L, \nabla) \to (\Lambda^{r}N^*L,D). $$
Schur's Lemma then implies the existence of isomorphic $\rH_x$-subrepresentations on each fibre. That is, there is a parallel isomorphism between certain rank-$m$ subbundles $\Lambda^r_m T^*L \subset \Lambda^r T^*L$ and $\Lambda^r_m N^*L\subset \Lambda^r N^*L$. (A parallel isomorphism is a bundle isomorphism that maps one connection to the other.)
\end{rmk}

Combining Theorem~\ref{thm:main} with the Gauss--Codazzi--Ricci relations of~\eqref{eq:GCR-general} yields the following.

\begin{cor} \label{cor:main-GCR}
Let $\alpha$ be a parallel-compliant calibration, and let $L$ be an $\alpha$-calibrated submanifold. Then we have
$$ [B_u, B_v] \in \frh \qquad \text{and} \qquad (\wh{\nabla}_u B)_v - (\wh{\nabla}_v B)_u \in \frh^{\perp_{\frg}} \qquad \text{for all $u, v \in \Gamma(TL)$.} $$
\end{cor}
\begin{proof}
Since $(M, g)$ has a torsion-free $\rG$-structure, the holonomy algebra $\hol(\ol{\nabla})$ lies in $\frg$. By the Ambrose--Singer Theorem, this gives $\ol{R}_{u,v} \in \frg$. From the three equations in~\eqref{eq:GCR-temp}, we see that
$$ \wh{R}_{u,v} \, , \, [B_u, B_v] \in \Gamma(\Lambda^2 TL \oplus \Lambda^2 NL), \qquad (\wh{\nabla}_u B)_v - (\wh{\nabla}_v B)_u \in \Gamma(TL \otimes NL). $$
Theorem~\ref{thm:main} gives $B_u, B_v \in \frg$, so $[B_u, B_v] \in \frg$, and also $\wh{R}_{u,v} \in \frh \subset \frg$. Thus at every point in $L$ we have
$$ \wh{R}_{u,v} \, , \, [B_u, B_v] \in \frg \cap (\Lambda^2 TL \oplus \Lambda^2 NL) = \frh. $$
But then by~\eqref{eq:GCR-general} we have
$$ (\wh{\nabla}_u B)_v - (\wh{\nabla}_v B)_u = \ol{R}_{u,v} - \wh{R}_{u,v} - [B_u, B_v] \in \frg, $$
so  we get
$$ (\wh{\nabla}_u B)_v - (\wh{\nabla}_v B)_u \in \frg \cap (TL \otimes NL) = \frh^{\perp_{\frg}}, $$
as claimed.
\end{proof}

\subsection{Application of the main theorem to examples} \label{sec:main-examples}

In this section we examine the consequences of Theorem~\ref{thm:main} for the interesting geometric examples of Section~\ref{sec:calibrated-review}, which we showed were all compliant in Section~\ref{sec:compliancy-examples}. The results from complex submanifolds of K\"ahler manifolds have been well-known for a long time. By contrast, the results for special Lagrangian, associative, coassociative, and Cayley submanifolds may be known to some experts, but the authors have not seen them explicitly stated anywhere in the literature.

\begin{ex}[Complex submanifolds] \label{ex:complex-main} Let $(M^{2m}, g, \omega)$ be a manifold with torsion-free $\U{m}$-structure, also known as a K\"ahler manifold. Let $L^{2p}$ be a submanifold calibrated by $\frac{1}{p!} \omega^p$. Here $\rG = \U{m}$ and $\rH = \U{p} \times \U{m-p}$. In this case Theorem~\ref{thm:main} says that:
\begin{itemize}
\item $\Hol(\wh{\nabla}) \subset \rH$. This is equivalent to $\Hol(\nabla) \subset \U{p} $, and $\Hol(D) \subset \U{m-p}$.
\item $B_u \diamond \omega^j = 0$ for all $j=1, \ldots, p$. This is equivalent to $B_u \diamond \omega = 0$, so $B_u \in \mathfrak{u}(m)$.
\end{itemize}
Equation~\eqref{eq:complex-forms} shows that $\omega = \omega^{2,0} + \omega^{0,2}$, so the form $B_u \diamond \omega$ is type $(1,1)$. Thus, we can characterize the second fundamental form condition $B_u \diamond \omega = 0$ by evaluating on a pair $v, \xi$ where $v \in \Gamma(TL)$ and $\xi \in \Gamma(NL)$. Using~\eqref{eq:B} we get
\begin{align*}
0 & = (B_u \diamond \omega)(v, \xi) = \omega(B_u v, \xi) + \omega(v, B_u \xi) \\
& = \omega( A_u v, \xi ) - \omega( v, A_u^* \xi ).
\end{align*}
Recalling that $\omega( \cdot, \cdot ) = g (J \cdot, \cdot)$ where $J$ is the associated orthogonal complex structure, we get
\begin{align*}
0 & = g (J A_u v, \xi) - g(Jv, A_u^* \xi) = g (J A_u v, \xi) - g(A_u Jv, \xi), \\
& = - g(v, A_u^* J \xi) + g(v, J A_u^* \xi),
\end{align*}
which in turn is equivalent to either of the equations
\begin{equation} \label{eq:superminimal}
\begin{aligned}
A_u (J v) & = J (A_u v), \\
A_u^* (J \xi) & = J (A_u^* \xi),
\end{aligned}
\qquad \text{ for all $u, v \in \Gamma(TL)$ and $\xi \in \Gamma(NL)$.}
\end{equation}
Note that this is exactly what is given by~\eqref{eq:VCP-case} for $k=2$ and $P = J$. These are all well-known facts about complex submanifolds of K\"ahler manifolds.
\end{ex}

\begin{rmk} \label{rmk:complex-main}
Since $\U{1} = \SO{2}$, if $p=1$ then the condition on $\Hol(\nabla)$ is automatic, while if $p = m-1$ then the condition on $\Hol(D)$ is automatic. In particular, for a submanifold $L^2$ of real dimension $2$ which is calibrated by $\omega$ in a K\"ahler manifold $M^4$ of real dimension $4$, the only nontrivial condition is the condition~\eqref{eq:superminimal} on the second fundamental form. Classically, such surfaces are called \emph{superminimal}. It is easy to check that~\eqref{eq:superminimal} implies the vanishing of the mean curvature, as it must because calibrated submanifolds are volume-minimizing. See also Theorem~\ref{thm:superminimal}.
\end{rmk}

\begin{ex}[Special Lagrangian submanifolds] \label{ex:slag-main}
Let $(M^{2m}, g, \omega, \Upsilon)$ be a manifold with torsion-free $\SU{m}$-structure, also known as a Calabi-Yau manifold. Let $L^m$ be a submanifold calibrated by $\real (\Upsilon)$. Here $\rG = \SU{m}$ and $\rH = \SO{m}$, diagonally embedded, as explained following Table~\ref{table:GH}. In this case Theorem~\ref{thm:main} says that:
\begin{itemize}
\item $\hol(\wh{\nabla}) \subset \frh$.
\item $B_u \diamond \omega = 0$ and $B_u \diamond \Upsilon = 0$. This is equivalent to $B_u \in \mathfrak{su}(m)$.
\end{itemize}
The holonomy condition is equivalent to $\Hol(\nabla) = \Hol(D)$ as subgroups of $\SO{m}$. This could also be deduced, by Remark~\ref{rmk:parallel-bundle-map}, from the parallel bundle isomorphism $(TL, \nabla) \cong (NL, D)$ determined by $\omega$. Under this isomorphism, we have an identification of curvature operators $R^{\nabla} \cong R^D$.

In this case, equation~\eqref{eq:SLAG-forms} shows that $\omega = \omega^{1,1}$, so the form $B_u \diamond \omega$ is type $(2,0) + (0,2)$. Thus, we can characterize the second fundamental form condition $B_u \diamond \omega = 0$ by evaluating on a pair $v, w \in \Gamma(TL)$ and on a pair $\xi, \eta \in \Gamma(NL)$. Using~\eqref{eq:B} we get
\begin{align*}
0 & = (B_u \diamond \omega)(v, w) = \omega(B_u v, w) + \omega(v, B_u w) \\
& = \omega( A_u v, w ) + \omega( v, A_u w ).
\end{align*}
Computing as in Example~\ref{ex:complex-main}, we get
\begin{equation*}
0 = g (J A_u v, w) + g(Jv, A_u w) = g (J A_u v, w) + g(A_u^* Jv, w), \\
\end{equation*}
and similarly on $\xi, \eta$, which yields the pair of equations
\begin{equation*}
\begin{aligned}
J (A_u v) & = - A_u^* (J v), \\
J (A_u^* \xi) & = - A_u (J \xi),
\end{aligned}
\qquad \text{ for all $u, v \in \Gamma(TL)$ and $\xi \in \Gamma(NL)$.}
\end{equation*}
From $B_u \diamond \Upsilon = 0$, we get further identities on the second fundamental form. For example, equation~\eqref{eq:SLAG-forms} shows that $\real (\Upsilon)$ is of type $(m,0) + (m-2,2) + \cdots$, so $B_u \diamond (\real (\Upsilon))$ is of type $(m-1,1) + (m-3,3) + \cdots$. Thus, one consequence of $B_u \diamond \Upsilon = 0$ can be expressed by evaluating $B_u \diamond (\real (\Upsilon))$ on $v_1, \ldots, v_{m-1}, \xi$ where $v_i \in \Gamma(TL)$ and $\xi \in \Gamma(NL)$. Using~\eqref{eq:B} we get
\begin{align*}
0 & = (B_u \diamond (\real(\Upsilon))(v_1, \ldots, v_{m-1}, \xi) \\
& = (\real \Upsilon)(B_u v_1, \ldots, v_{m-1}, \xi) + \cdots + (\real \Upsilon)(v_1, \ldots, B_u v_{m-1}, \xi) + (\real \Upsilon)(v_1, \ldots, v_{m-1}, B_u \xi) \\
& = (\real \Upsilon)(A_u v_1, \ldots, v_{m-1}, \xi) + \cdots + (\real \Upsilon)(v_1, \ldots, A_u v_{m-1}, \xi) - (\real \Upsilon)(v_1, \ldots, v_{m-1}, A^*_u \xi).
\end{align*}
The other consequences are obtained similarly.
\end{ex}

Although not strictly necessary, we make a brief digression to discuss the decomposition of curvature for connections on rank-$4$ bundles in terms of quaternions. This allows us to give another demonstration of the consequences of Theorem~\ref{thm:main} for $\hol(\wh{\nabla})$ in the associative, coassociative, and Cayley cases.

As usual, we view $\SO{3} = \Sp{1} / \{\pm 1\}$ acting on $\R^3 = \imag(\Qu)$ as $([h], x) \mapsto h x \ol{h}$. Under the natural identifications $\imag(\Qu) = \mathfrak{sp}(1)$, and $\Lambda^2 (\R^3)^* = \mathfrak{so}(3)$, we have the induced isomorphism
\begin{equation} \label{eq:SO(3)-ident}
\imag(\Qu) \to \Lambda^2 (\R^3)^*, \qquad q \mapsto 2 \st q^{\flat},
\end{equation}
where $q^{\flat}$ is the $1$-form $q^{\flat}(y) = \langle q, y \rangle$, and $\st$ denotes the Hodge star on $\R^3$.

We also view $\SO{4} = \Sp{1}^2 / \{\pm 1\}$ acting on $\R^4 = \Qu$ as $([h_1,h_2], x) \longmapsto h_1 x \ol{h}_2$. This action, together with the natural identifications $\imag(\Qu) = \mathrm{Lie}(\Sp{1} \times \{1\})$, $\imag(\Qu) = \mathrm{Lie}(\{1\} \times \Sp{1})$, and $\Lambda^2 (\R^4)^* = \mathfrak{so}(4)$, yield the isomorphisms
\begin{equation} \label{eq:SO(4)-ident}
\begin{aligned}
\imag(\Qu) & = \mathrm{Lie}(\Sp{1} \times \{1\}) \to \Lambda^2_+ (\R^4)^*, \qquad q \mapsto \gamma_q^+(x,y) = \langle qx, y \rangle, \\
\imag(\Qu) & = \mathrm{Lie}(\{1\} \times \Sp{1}) \to \Lambda^2_- (\R^4)^*, \qquad q \mapsto \gamma_q^-(x,y) = \langle x\ol{q}, y \rangle.
\end{aligned}
\end{equation}
Let $h \in \Sp{1}$ and consider its action on $\Lambda^2 (\R^4)^*$ via pullback under left and right multiplications on $\Qu$, that we denote by $l_h$ and $r_h$. We have
\begin{equation}
l_h^* \gamma_q^+ = \gamma_{\ol{h} q h}^+, \qquad r_h^* \gamma_q^+ = \gamma_q^+, \qquad l_h^* \gamma_q^- = \gamma_q^-, \qquad r_h^* \gamma_q^- = \gamma_{h q \ol{h}}^-.
\end{equation}
Therefore the induced action of $\Sp{1}^2 / \{\pm 1 \}$ on $\Lambda^2 (\R^4)$ by pullback is
\begin{equation} \label{eq:SO(4)-Lambda2}
[h_1, h_2] \cdot (\gamma_{q_1}^+ + \gamma_{q_2}^-) = \gamma_{\ol{h}_1 q_1 h_1}^+ + \gamma_{\ol{h}_2 q_2 h_2}^-.
\end{equation}

Let $E \to L$ be real rank-$4$ oriented vector bundle endowed with a fibre metric and a metric-compatible connection $\wt{\nabla}$. Then $\hol (\wt{\nabla})$ decomposes into two subalgebras $\hol_{\pm} (\wt{\nabla})$ according to the splitting $\mathfrak{so}(4)= \Lambda^2_+\R^4 \oplus \Lambda^2_-\R^4$. The differential of the action described in~\eqref{eq:SO(4)-Lambda2} is $-\mathrm{ad}$, and it provides an identification between $\hol(\wt{\nabla})$ and $\hol(\Lambda^2 E^*)$ (see~\cite[Proposition 2.3.7]{Joyce}). Taking into account the splitting $\Lambda^2 E^* = \Lambda^2_+ E^* \oplus \Lambda^2_- E^*$, we have
\begin{equation} \label{eq:holonomy-E}
\hol(\Lambda^2_{\pm} E^*) = -\mathrm{ad} (\hol_{\pm}(\wt{\nabla})).
\end{equation}
 The curvature $\wt{R} \in \Gamma(\Lambda^2 TL \otimes \Lambda^2 E^*)$ of $\wt{\nabla}$ also decomposes as $\wt{R} = \wt{R}_+ + \wt{R}_-$ with $\wt{R}_{\pm} \in \Gamma(\Lambda^2 T^*L \otimes \Lambda^2_\pm E^*)$. These components are related to the curvature $\wt{F}_{\pm}$ of the bundles $\Lambda^2_{\pm} E^*$ (see~\cite[Proposition 2.1.9]{Joyce}) via the equation
\begin{equation}\label{eq:curvature-E}
\wt{F}_{\pm} = -\mathrm{ad} (\wt{R}_{\pm}).
\end{equation}

These identifications are useful if $L$ is $4$-dimensional, so that $E = TL$ with $\wt{\nabla} = \nabla$, and if $NL$ is a rank-$4$ vector bundle with $\wt{\nabla} = D$.

\begin{ex}[Associative submanifolds] \label{ex:assoc-main}
Let $(M^7, g, \ph, \ps = \st \ph)$ be a manifold with torsion-free $\G$-structure, also known as a torsion-free $\G$-manifold. Let $L^3$ be a submanifold calibrated by $\ph$. Here $\rG = \G$ and $\rH = \SO{4}$, embedded as explained following Table~\ref{table:GH}. In this case Theorem~\ref{thm:main} says that:
\begin{itemize}
\item $\hol(\wh{\nabla}) \subset \frh$. 
\item $B_u \diamond \ph = 0$. This is equivalent to $B_u \in \frg_2$.
\end{itemize}
To understand the holonomy condition, recall that $\rH = \SO{4}$ acts on $\R^7 = \imag (\Qu) \oplus \Qu$ by 
$$ [h_1, h_2] (x,y) = (h_1 x \ol{h}_1, h_1 y \ol{h}_2), \qquad [h_1, h_2] \in \Sp{1}^2 / \{\pm 1 \}, \quad (x,y) \in \imag (\Qu) \oplus \Qu. $$
Combining~\eqref{eq:SO(3)-ident} and~\eqref{eq:SO(4)-ident} we obtain
$$ \frh = \{ (2 \st q^{\flat}, \gamma_q^+), q \in \imag (\Qu)\} \oplus \Lambda^2_- \Qu^*. $$
This above also appears in~\cite{BM}. (Aside: it is shown directy in~\cite{BM} that $\frh^{\perp_{\frg_2}} = \frg_2 \cap( \mathrm{Im}(\Qu)\otimes \Qu)$, giving another demonstration of compliancy in this case.) Thus the condition $\hol(\wh{\nabla}) \subset \frh$ yields $\hol_+ (D) \cong \hol(\nabla)$. In addition, the component $R^D_+$ of the curvature on the normal bundle is determined by $R^{\nabla}$.

This could also be deduced, by Remark~\ref{rmk:parallel-bundle-map}, from the parallel bundle isomorphism $(\Lambda^2 T^*L, \nabla) \cong (\Lambda^2_+ N^*L, D)$ determined by the $(1, 2)$ component of $\ph$, along with equations~\eqref{eq:holonomy-E} and~\eqref{eq:curvature-E}.

For the second fundamental form condition,~\eqref{eq:assoc-form} shows that $\ph$ is of type $(3,0) + (1,2)$, so $B_u \diamond \ph$ is of type $(2,1) + (0,3)$. Thus, there are two consequences of $B_u \diamond \ph = 0$. The first can be expressed by evaluating $B_u \diamond \ph$ on $v_1, v_2, \xi$ where $v_i \in \Gamma(TL)$ and $\xi \in \Gamma(NL)$. Using~\eqref{eq:B} one computes as before that
\begin{equation*}
0 = \ph(A_u v_1, v_2, \xi) + \ph(v_1, A_u v_2, \xi) - \ph(v_1, v_2, A^*_u \xi).
\end{equation*}
Similarly the $(0,3)$ component of $B_u \diamond \ph = 0$ yields
$$ \ph(A_u^* \xi_1, \xi_2, \xi_3) + \ph(\xi_1, A_u^* \xi_2, \xi_3) + \ph(\xi_1, \xi_2, A_u^* \xi_3) = 0. $$
Equivalently, by~\eqref{eq:VCP-case} with $k=2$ and $P = \times$ is the cross product given by $\ph(X_1, X_2, X_3) = g(X_1 \times X_2, X_3)$, these conditions can be repackaged as
\begin{align} \label{eq:assoc-2ndff}
A_u (v_1 \times v_2) & = (A_u v_1) \times v_2 + v_1 \times (A_u v_2), \\ \nonumber
- A_u^* (v_1 \times \xi_1) & = (A_u v_1) \times \xi_1 - v_1 \times (A_u^* \xi_1), \\ \nonumber
A_u (\xi_1 \times \xi_2) & = - (A_u^* \xi_1) \times \xi_2 - \xi_1 \times (A_u^* \xi_2),
\end{align}
for $u, v_1,v_2 \in \Gamma(TL)$, $\xi_1, \xi_2 \in \Gamma(NL)$.
\end{ex}

\begin{ex}[Coassociative submanifolds] \label{ex:coassoc-main}
Let $(M^7, g, \ph, \ps = \st \ph)$, $\rG = \G$, and $\rH = \SO{4}$ be as in Example~\ref{ex:assoc-main}. Let $L^4$ be a submanifold calibrated by $\ps$. In this case Theorem~\ref{thm:main} again gives:
\begin{itemize}
\item $\hol(\wh{\nabla}) \subset \frh$. 
\item $B_u \diamond \ps = 0$. This is equivalent to $B_u \in \frg_2$.
\end{itemize}
However, this time the roles of $TL$ and $NL$ are reversed, so the holonomy condition $\hol(\wh{\nabla}) \subset \frh$ tells us that $\hol(D) = \hol_+ (\nabla)$. In addition, the curvature $R^D$ of the normal bundle is determined by $R^{\nabla}_+$. It is well-known that on a $4$-manifold, $R^{\nabla}_+$ depends on the scalar curvature $s$, the self-dual part $W_+$ of the Weyl tensor, and the traceless part of the Ricci endomorphism $\mathrm{Ric}^0$.

Again, this could also be deduced, by Remark~\ref{rmk:parallel-bundle-map}, from the parallel bundle isomorphism $(\Lambda^2_+ T^*L, \nabla) \cong (\Lambda^2 N^*L, D)$ determined by the $(2, 2)$ component of $\ps$, along with equations~\eqref{eq:holonomy-E} and~\eqref{eq:curvature-E}.

For the second fundamental form condition,~\eqref{eq:coassoc-form} shows that $\ps$ is of type $(4,0) + (2,2)$, so $B_u \diamond \ps$ is of type $(3,1) + (1,3)$. Thus, there are two consequences of $B_u \diamond \ps = 0$. The first can be expressed by evaluating $B_u \diamond \ps$ on $v_1, v_2, v_3, \xi$ where $v_i \in \Gamma(TL)$ and $\xi \in \Gamma(NL)$. Using~\eqref{eq:B} one computes as before that
\begin{equation*}
0 = \ps(A_u v_1, v_2, v_3, \xi) + \ps(v_1, A_u v_2, v_3, \xi) + \ps(v_1, v_2, A_u v_3, \xi) - \ps(v_1, v_2, v_3, A^*_u \xi).
\end{equation*}
Similarly the $(1,3)$ component of $B_u \diamond \ps = 0$ yields
$$ \ps(A_u v_1, \xi_1, \xi_2, \xi_3) - \ps(v_1, A_u^* \xi_1, \xi_2, \xi_3) - \ps(v_1, \xi_1, A_u^* \xi_2, \xi_3) - \ps(v_1, \xi_1, \xi_2, A_u^* \xi_3) = 0. $$
Equivalently, from the fact an element $B \in \mathfrak{so}(7)$ lies in $\frg_2$ if and only if $B(X \times Y) = (BX) \times Y + X \times (BY)$ for all $X, Y \in \R^7$, the condition $B_u \in \frg_2$ can also be repackaged as
\begin{align*}
- A_u^* (v_1 \times v_2) & = (A_u v_1) \times v_2 + v_1 \times (A_u v_2), \\
A_u (v_1 \times \xi_1) & = (A_u v_1) \times \xi_1 - v_1 \times (A_u^* \xi_1), \\
- A_u^* (\xi_1 \times \xi_2) & = - (A_u^* \xi_1) \times \xi_2 - \xi_1 \times (A_u^* \xi_2),
\end{align*}
for $u, v_1, v_2 \in \Gamma(TL)$, $\xi_1, \xi_2 \in \Gamma(NL)$.
\end{ex}

\begin{ex}[Cayley submanifolds] \label{ex:Cayley-main}
Let $(M^8, g, \Ph)$ be a manifold with torsion-free $\Spin{7}$-structure, also known as a torsion-free $\Spin{7}$-manifold. Let $L^4$ be a submanifold calibrated by $\Ph$. Here $\rG = \Spin{7}$ and $\rH = \Sp{1}^3 / \{ \pm 1 \}$, embedded as explained following Table~\ref{table:GH}. In this case Theorem~\ref{thm:main} says that:
\begin{itemize}
\item $\hol(\wh{\nabla}) \subset \frh$. 
\item $B_u \diamond \Ph = 0$. This is equivalent to $B_u \in \mathfrak{spin}(7)$.
\end{itemize}
To understand the holonomy condition, recall that $\rH = \Sp{1}^3 / \{\pm 1 \}$ acts on $\R^8 = \Qu \oplus \Qu$ by 
$$ [h_1, h_2, h_3](x,y) = (h_1 x \ol{h}_2, h_1 x \ol{h}_3), \qquad [h_1, h_2, h_3] \in \Sp{1}^3 / \{\pm 1 \}, \quad (x, y) \in \Qu \oplus \Qu. $$
Denote $\Qu_1 = \Qu \oplus \{0\}$ and $\Qu_2 = \{0\} \oplus \Qu$. From~\eqref{eq:SO(4)-ident} we obtain
$$ \frh = \{ (\gamma_q^+, \gamma_q^+), q \in \imag(\Qu) \} \oplus \Lambda^2_- \Qu_1^* \oplus \Lambda^2_- \Qu_2^*. $$
The above also appears in~\cite{BM}. (Aside: it is shown directy in~\cite{BM} that $\frh^{\perp_{\mathfrak{spin}(7)}} = \mathfrak{spin}(7) \cap( \Qu_1 \otimes \Qu_2)$, giving another demonstration of compliancy in this case.) Thus the condition $\hol(\wh{\nabla}) \subset \frh$ yields $\hol_+ (D) = \hol_+ (\nabla)$. In addition, the component $R^D_+$ of the curvature on the normal bundle is determined by $R^{\nabla}_+$. As in the coassociative case, $R^{\nabla}_+$ depends on the scalar curvature $s$, the self-dual part $W_+$ of the Weyl tensor, and the traceless part of the Ricci endomorphism $\mathrm{Ric}^0$.

Again, this could also be deduced, by Remark~\ref{rmk:parallel-bundle-map}, from the parallel bundle isomorphism $(\Lambda^2_+ T^*L, \nabla) \cong (\Lambda^2_+ N^*L, D)$ determined by the $(2, 2)$ component of $\Ph$, along with equations~\eqref{eq:holonomy-E} and~\eqref{eq:curvature-E}.

For the second fundamental form condition,~\eqref{eq:Cayley-form} shows that $\Ph$ is of type $(4,0) + (2,2) + (0,4)$, so $B_u \diamond \Ph$ is of type $(3,1) + (1,3)$. Thus, there are two consequences of $B_u \diamond \Ph = 0$. The first can be expressed by evaluating $B_u \diamond \Ph$ on $v_1, v_2, v_3, \xi$ where $v_i \in \Gamma(TL)$ and $\xi \in \Gamma(NL)$. Using~\eqref{eq:B} one computes as before that
\begin{equation*}
0 = \Ph(A_u v_1, v_2, v_3, \xi) + \Ph(v_1, A_u v_2, v_3, \xi) + \Ph(v_1, v_2, A_u v_3, \xi) - \Ph(v_1, v_2, v_3, A^*_u \xi).
\end{equation*}
Similarly the $(1,3)$ component of $B_u \diamond \ph = 0$ yields
$$ \Ph(A_u v_1, \xi_1, \xi_2, \xi_3) - \Ph(v_1, A_u^* \xi_1, \xi_2, \xi_3) - \Ph(v_1, \xi_1, A_u^* \xi_2, \xi_3) - \Ph(v_1, \xi_1, \xi_2, A_u^* \xi_3) = 0. $$
Equivalently, by~\eqref{eq:VCP-case} with $k=3$ and $P$ is the $3$-fold cross product given by $\Ph(X_1, X_2, X_3 X_4) = g(P(X_1, X_2, X_3), X_4)$, these conditions can be repackaged as
\begin{align*}
A_u ( P(v_1, v_2, v_3)) & = P(A_u v_1, v_2, v_3) + P(v_1, A_u v_2, v_3) + P(v_1, v_2, A_u v_3), \\
- A_u^* ( P(v_1, v_2, \xi_1)) & = P(A_u v_1, v_2, \xi_1) + P(v_1, A_u v_2, \xi_1) - P(v_1, v_2, A_u^* \xi_1), \\
A_u( P(v_1, \xi_1, \xi_2)) & = P(A_u v_1, \xi_1, \xi_2) - P(v_1, A_u^* \xi_1, \xi_2) - P(v_1, \xi_1, A_u^* \xi_2), \\
- A_u^* ( P(\xi_1, \xi_2, \xi_3)) & = - P(A_u^* \xi_1, \xi_2, \xi_3) - P(\xi_1, A_u^* \xi_2, \xi_3) - P(\xi_1, \xi_2, A_u^* \xi_3),
\end{align*}
for $u, v_1,v_2,v_3 \in TL$, $\xi_1, \xi_2, \xi_3 \in NL$.
\end{ex}

\subsection{Infinitesimally $\alpha$-calibrated submanifolds} \label{sec:inf-calib}

Our main result, Theorem~\ref{thm:main}, motivates the following discussion. Let $(n, k, \alpha, \rG, \rH)$ be one of the rows from Table~\ref{table:GH}, so $\rG \subset \SO{n}$ and the embedding $\rH \subset \rG$ is as described following the table. In each case, we in fact have $\rH = \rG \cap (\SO{k} \times \SO{n-k})$.

Let $(M^n, g)$ be an oriented Riemannian manifold of dimension $n$. We do \emph{not} assume that $(M, g)$ is equipped with a calibration $k$-form $\alpha$. Let $L^k$ be an oriented $k$-dimensional submanifold of $M$. Let $\wh{\nabla} = \nabla \oplus D$ be the induced connection on $E = \rest{TM}{L} = TL \oplus NL$. Let $B$ denote the packaging of second fundamental form $A$ of $L$ in $M$ as given by~\eqref{eq:B}.

\begin{defn} \label{defn:H-str}
We say that $L$ has an $\rH$-structure if there is a reduction of the structure group of the frame bundle of $E$ from $\SO{k} \times \SO{n-k}$ to $\rH$. This is equivalent to the existence of a $\rG$-structure on the $\R^n$-bundle $E$ over $L$ for which each fibre of $TL$ is $\alpha$-calibrated. Note that the $k$-form $\alpha$ exists \emph{only along $L$}. That is, we have a section $\alpha \in \Gamma( \Lambda^k E^* ) = \oplus_{p+q=k} \Gamma \big( (\Lambda^p T^*L) \otimes (\Lambda^q N^*L) \big)$, but we do \emph{not} assume that $\alpha$ is the pullback by the immersion $\iota \colon L \to M$ of a $k$-form on $M$.
\end{defn}

\begin{defn} \label{defn:inf-calib}
Suppose that $L$ is equipped with an $\rH$-structure as in Definition~\ref{defn:H-str}. Then we say that $L$ is \emph{infinitesimally $\alpha$-calibrated} if
$$ \hol(\wh{\nabla}) \subset \frh \qquad \text{and} \qquad B_u \in \frg \, \text{ for all $u \in \Gamma(TL)$.} \qedhere $$
\end{defn}

By Theorem~\ref{thm:main}, these are precisely the two conditions on the extrinsic geometric data $(A, \nabla, D)$ that hold if $L$ is in fact calibrated with respect to a parallel-compliant calibration $\alpha$.

We claim that an infinitesimally $\alpha$-calibrated submanifold is necessarily \emph{minimal}. That is, it has vanishing mean curvature $H$. This can be checked directly in each of the five cases of Table~\ref{table:GH} using the condition $B_u \in \frg$. We demonstrate with two specific examples.

In the case of an infinitesimally $\frac{1}{p!} \omega^p$-calibrated submanifold, we have a $\U{p} \times \U{m-p}$-structure on $E$. This means we have orthogonal complex structures $J_T$ and $J_N$ on $TL$ and $NL$, respectively, and thus an orthogonal complex structure $J = J_T \oplus J_N$ on $E$. From~\eqref{eq:superminimal} we have
$$ A_u (J v) = J (A_u v) \quad \text{for $u, v \in \Gamma(TL)$}. $$
The above gives, using~\eqref{eq:2ndff}, that for any $\xi \in \Gamma(NL)$ we have
\begin{align*}
\langle A(Ju, Ju), \xi \rangle & = \langle A_{Ju} (Ju), \xi \rangle = \langle J (A_{Ju} u), \xi \rangle = - \langle A_{Ju} u, J \xi \rangle \\
& = - \langle A_u (Ju), J \xi \rangle = - \langle J (A_u u), J \xi \rangle = - \langle A_u u, \xi \rangle = - \langle A(u, u), \xi \rangle.
\end{align*}
Thus for a local orthonormal frame $\{ e_1, J e_1, \ldots, e_p, J e_p \}$ of $L$ we obtain
$$ \langle H, \xi \rangle = \sum_{i=1}^p \langle A(e_i, e_i), \xi \rangle + \sum_{i=1}^p \langle A(J e_i, J e_i), \xi \rangle = 0, $$
as claimed.

In the case of an infinitesimally $\ph$-calibrated submanifold, an $\rH = \SO{4}$-structure on $E$ is equivalent to a $\G$-structure on the $\R^7$-bundle $E$ over $L$ for which each fibre of $TL$ is an associative subspace, and thus each fibre of $NL$ is a coassociative subspace. Thus we have a $2$-fold cross product $\times$ on the fibres of $E$. From~\eqref{eq:assoc-2ndff}, we have
$$ A_u (v_1 \times v_2) = (A_u v_1) \times v_2 + v_1 \times (A_u v_2). $$
The above gives, using~\eqref{eq:2ndff}, that for any $\xi \in \Gamma(NL)$ we have
\begin{align*}
\langle A(u \times v, u \times v), \xi \rangle & = \langle A_{u \times v}(u \times v), \xi \rangle = \langle (A_{u \times v} u) \times v + u \times (A_{u \times v}) v, \xi \rangle \\
& = \langle A_{u \times v} u , v \times \xi \rangle - \langle A_{u \times v} v, u \times \xi \rangle = \langle A_u (u \times v), v \times \xi \rangle - \langle A_v (u \times v), u \times \xi \rangle \\
& = \langle (A_u u) \times v + u \times (A_u v), v \times \xi \rangle - \langle (A_v u) \times v + u \times (A_v v), u \times \xi \rangle.
\end{align*}
Using the identity $\langle X \times Y, Z \times W \rangle = \langle X \wedge Y, Z \wedge W \rangle + \ps (X, Y, Z, W)$ (see~\cite[Remark 3.72]{K-intro}, where the opposite orientation is used), then if $u, v$ are \emph{orthonormal} the above becomes
\begin{align*}
\langle A(u \times v, u \times v), \xi \rangle & = \langle (A_u u) \wedge v, v \wedge \xi \rangle + \ps(A_u u, v, v, \xi) + \langle u \wedge (A_u v), v \wedge \xi) + \ps(u, A_uv, v, \xi) \\ & \qquad {} - \langle (A_v u) \wedge v, u \wedge \xi \rangle - \ps(A_v u, v, u, \xi) - \langle u \wedge (A_v v), u \wedge \xi) - \ps(u, A_v v, u, \xi) \\
& = - \langle A_u u, \xi \rangle - \langle A_v v, \xi \rangle.
\end{align*}
We can choose a local orthonormal frame of the form $\{ e_1, e_2, e_1 \times e_2 \}$.  Thus we have
$$ \langle H, \xi \rangle = \langle A_{e_1} e_1, \xi \rangle + \langle A_{e_2} e_2, \xi \rangle + \langle A_{e_1 \times e_2} (e_1 \times e_2), \xi \rangle = 0, $$
as claimed.

\begin{rmk} \label{rmk:necess-minimal}
One might hope that there exists a general argument demonstrating that $B_u \in \frg$ forces the trace of $A$ to vanish. In fact, in the cases where $\alpha$ arises from a vector cross product, which are the cases $\alpha = \omega, \ph, \Ph$, then one can give a unified proof that infinitesimally $\alpha$-calibrated implies minimal. It might be possible to extend this argument to include the special Lagrangian and coassociative cases, or more general compliant calibrations. (See, for example,~\cite[Proposition 2.13]{CKM} for a result which shows that the coassociative calibration behaves similarly to calibrations defined by vector cross products.)

However, this question could also be seen as analogous to the following. For Riemannian manifolds with special holonomy (other than K\"ahler), one can show on a case-by-case basis from Berger's list that the holonomy algebra constraint on the curvature tensor forces the metric to be Einstein. There does not seem to be a way to prove this in a unified way. However, for those special holonomy metrics which are not K\"ahler or quaternionic-K\"ahler, it is known that they admit parallel spinors, and then there is a spin-geometric argument that establishes Ricci-flatness. See~\cite{F, W} for details. Thus there may be a spin-geometric analogue in this case.
\end{rmk}

A natural question that arises is the following.

\begin{quest} \label{quest:inf-calib}
Suppose that $L^k$ is an infinitesimally $\alpha$-calibrated submanifold of $(M^n, g)$. Under what conditions can be conclude that there exists a globally defined parallel-compliant calibration $\alpha$ on $M$ such that $L$ is (genuinely) $\alpha$-calibrated?
\end{quest}

Of course, if $L$ is calibrated with respect to a parallel (and hence closed) calibration form $\alpha$, then not only must $L$ be minimal, but it must also be \emph{stable}, since calibrated submanifolds are homologically volume minimizing. Thus, although an infinitesimally $\alpha$-calibrated submanifold is indeed always minimal, we need to further assume that it is also stable to have any hope of answering Question~\ref{quest:inf-calib}. (Here stable means that the second variation of the volume functional at such a critical point is nonnegative.) A classical result along these lines is the following theorem (stated somewhat imprecisely, but including our annotations to relate the result to the present paper.)

\begin{thm}[Micallef~\cite{M}] \label{thm:superminimal}
Let $L^2$ be an oriented surface in $\R^4$. Assume that $L^2$ is a \emph{complete and stable} minimal surface. [In this case, by Remark~\ref{rmk:complex-main}, the surface $L$ \emph{always} admits a canonical (unique) $\rH = \U{1} \times \U{1}$-structure.] If $L^2$ is infinitesimally $\omega$-calibrated [which Micallef calls \emph{superminimal}], and if a further technical assumption also holds, then there exists a global parallel (hence integrable) orthogonal complex structure $J$ on $\R^4$ with respect to which $L$ is a complex submanifold. [That is, there is a K\"ahler structure $\omega$ on $\R^4$ compatible with the Euclidean metric such that $L$ is (genuinely) $\omega$-calibrated.]
\end{thm}

Taking these observations into account, a more refined version of Question~\ref{quest:inf-calib} is the following.

\begin{quest} \label{quest:inf-calib-2}
Let $L^k$ be an oriented submanifold in $(M^n, g)$. Suppose that $L^k$ is a \emph{complete and stable} minimal submanifold. Assume that $L$ admits an $\rH$-structure, for some row $(n, k, \alpha, \rG, \rH)$ of Table~\ref{table:GH}. If $L$ is infinitesimally $\alpha$-calibrated, then do there exist further technical hypotheses which would ensure that there exists a globally defined parallel-compliant calibration $\alpha$ on $M$ such that $L$ is (genuinely) $\alpha$-calibrated?
\end{quest}

\end{document}